\documentclass[11pt]{amsart}
\usepackage{amssymb,nicefrac,cite,bm}
\usepackage[mathscr]{eucal}
\usepackage{color}

\numberwithin{equation}{section}

\newtheorem{theorem}{Theorem}[section]
\newtheorem{lemma}{Lemma}[section]

\newtheorem{corollary}{Corollary}[section]

\theoremstyle{definition}

\theoremstyle{remark}
\newtheorem{remark}{Remark}[section]
\newtheorem{notation}{Notation}[section]

\newcommand{\df}{\;{:=}\;}

\DeclareMathOperator*{\trace}{tr}
\DeclareMathOperator*{\interior}{int}
\newcommand{\LA}{\mathcal{L}}
\newcommand{\PA}{\mathcal{P}}
\newcommand{\cG}{\mathcal{G}}
\newcommand{\cM}{\mathcal{M}}
\newcommand{\cV}{\mathscr{V}}
\newcommand{\cH}{\mathcal{H}}
\newcommand{\cK}{\mathcal{K}}
\newcommand{\cX}{{\mathscr{X}}}
\newcommand{\cO}{{\mathscr{O}}}
\newcommand{\RR}{\mathbb{R}}
\newcommand{\NN}{\mathbb{N}}
\newcommand{\ua}{{\underline\alpha}}
\newcommand{\oa}{{\overline\alpha}}
\newcommand{\abs}[1]{\lvert#1\rvert}

\newcommand{\norm}[1]{\lVert#1\rVert}
\newcommand{\bnorm}[1]{\bigl\lVert#1\bigr\rVert}

\newcommand{\ucZ}{\underline{\mathcal{Z}}}
\newcommand{\ocZ}{\overline{\mathcal{Z}}}

\begin{document}

\title[RISK-SENSITIVE CONTROL AND A COLLATZ--WIELANDT FORMULA]
{RISK-SENSITIVE CONTROL AND AN ABSTRACT\\[3pt]
COLLATZ--WIELANDT FORMULA}

\author{Ari\ Arapostathis}
\address{Department of Electrical and Computer Engineering,
The University of Texas at Austin,
1 University Station, Austin, TX 78712}
\email{ari@ece.utexas.edu}
\thanks{Supported in part by the Office of Naval Research
through the Electric Ship Research and Development Consortium.}

\author{Vivek S.\ Borkar}
\address{Department of Electrical Engineering, Indian Institute of Technology,
Powai, Mumbai 400076, India}
\email{borkar.vs@gmail.com}
\thanks{Supported in part by grant 11IRCCSG014 from IIT Bombay and a
J.\ C.\ Bose Fellowship
from the Department of Science and Technology, Government of India.}

\author{K.\ Suresh Kumar}
\address{Department of Mathematics, Indian Institute of Technology,
Powai, Mumbai 400076, India}
\email{suresh@math.iitb.ac.in}
\thanks{Supported in part by DST project: SR/S4/MS:751/12}

\begin{abstract}
The `value' of infinite horizon risk-sensitive control is the principal
eigenvalue of a certain positive operator.
For the case of compact domain, Chang has built upon a nonlinear version of
the Krein--Rutman theorem to give a 'min-max' characterization of this
eigenvalue which may be viewed as a generalization of the classical
Collatz--Wielandt formula for the Perron--Frobenius eigenvalue of a
non-negative irreducible matrix.
We apply this formula to the Nisio semi group associated with risk-sensitive
control and derive a variation all characterization of the optimal
risk-sensitive cost.
For the linear, i.e., uncontrolled case, this is seen to reduce to the
celebrated Donsker-Varadhan formula for principal eigenvalue of a second
order elliptic operator.
\end{abstract}

\maketitle

\medskip


\section{Introduction}

We consider the infinite horizon risk-sensitive control problem for a
controlled reflected diffusion in a bounded domain.
This seeks to minimize the asymptotic growth rate of the expected
`exponential of integral' cost, which in turn coincides with the principal
eigenvalue of a quasi-linear elliptic operator defined as the
pointwise envelope of a family of linear elliptic operators parametrized by
the `control' parameter.
The Kre{\u\i}n-Rutman theorem has been widely applied to study the time-asymptotic
behavior of linear parabolic equations \cite[Chapter~7]{Smith-95}.
A recent extension of the Kre{\u\i}n-Rutman theorem to positively 1-homogeneous
compact (nonlinear)
operators and the ensuing variational formulation for the positive
eigenpair extends the classical Collatz--Wielandt formula
for the Perron-Frobenius eigenvalue of irreducible non-negative matrices.
Using this, we are able to obtain a variational formulation for the
positive eigenpair that reduces to
the celebrated Donsker--Varadhan characterization thereof in the linear case.
In the linear case, the eigenvalue in the positive eigenpair coincides
with the principal eigenvalue.
This is not in general true for the nonlinear
case. Hence we obtain a Collatz-Wielandt formula for the unique
positive eigenpair (see the example in Remark~\ref{R4.2}).
This establishes interesting connections between  theory of
risk-sensitive control,
nonlinear Kre{\u\i}n-Rutman theorem, and  Donsker--Varadhan theory.

\section{Risk-sensitive control}\label{S2}
Let $Q \subset \mathbb{R}^d$ be an open bounded domain with a
$C^{3}$ boundary $\partial{Q}$  and  $\Bar{Q}$ denote its closure.
Consider a reflected controlled diffusion $X(\cdot)$ taking values in the
bounded domain $\Bar{Q}$ satisfying
\begin{equation}\label{sde}
\begin{split}
dX(t) &\;=\; b(X(t),v(t))\,dt + \sigma(X(t))\,dW(t)
- \gamma(X(t))\,d\xi(t)\,,\\[5pt]
d\xi(t) &\;=\; I\{X(t) \in \partial{Q}\}\,d\xi(t)
\end{split}
\end{equation}
for $t \ge 0$, with $X(0) = x$ and $\xi(0) = 0$. Here:
\begin{itemize}
\item[(a)] $b : \Bar{Q}\times\cV \to \mathbb{R}^d$ for a prescribed
compact metric control space $\cV$ is continuous and Lipschitz in its first
argument uniformly with respect to the second,

\item[(b)]
$\sigma : \Bar{Q} \to \mathbb{R}^{d\times d}$ is continuously
differentiable, its derivatives are H\"older continuous
with exponent $\beta_{0}>0$, and is uniformly non-degenerate in
the sense that the minimum eigenvalue of
\begin{equation*}
a(x)\;=\;[[a_{ij}(x)]]\;\df\;\sigma(x)\sigma^T(x)
\end{equation*}
is bounded away from zero.

\item[(c)]
$\gamma : \mathbb{R}^d \to \mathbb{R}^d$ is co-normal, i.e.,
$\gamma(x) = [\gamma_{1}(x), \dotsc, \gamma_d(x)]^T$, where
\begin{equation*}
\gamma_i(x) \;=\; \sum_{i=1}^da_{ij}(x)n_{j}(x)\,,\quad
x\in\partial Q\,,
\end{equation*}
$n(x)= [n_{1}(x), \dotsc, n_d(x)]^T$ is the unit outward normal.

\item[(d)]
$W(\cdot)$ is a $d$-dimensional standard Wiener process,

\item[(e)]
$v(\cdot)$ is a $\cV$-valued measurable process satisfying the
non-anticipativity condition:
for $t > s \ge 0$, $W(t) - W(s)$ is independent of $\{v(y), W(y): y \le s\}$.
A process $v$ satisfying this property is called an `admissible control'.
\end{itemize}

Let $r : \Bar{Q}\times\cV\to\mathbb{R}_{+}$ be a continuous
`running cost' function which is Lipschitz in its first argument uniformly with
respect to the second.
We define
\begin{equation*}
r_{\rm{max}}\;\df\;\max_{(x,v)\in\Bar{Q}\times\cV}|r(x,v)|\,.
\end{equation*}

The infinite horizon risk-sensitive problem aims to minimize the cost
\begin{equation}\label{risk}
\limsup_{T\uparrow\infty}\;
\frac{1}{T}\log E\left[e^{\int_{0}^Tr(X(s), v(s))ds}\right]\,,
\end{equation}
i.e., the mean asymptotic growth rate of the exponential of the total cost.
See \cite{Whit} for background and motivation.

We define
\begin{equation}\label{E-gen}
\begin{split}
\cG f(x) &\;\df\; \frac{1}{2}\trace\left(a(x)\nabla^{2}f(x)\right)
+ \cH\bigl(x,f(x),\nabla f(x)\bigr)\,, \ \mbox{where},\\[5pt]
\cH(x,f,p) &\;\df\;
\min_{v\in\cV}\; \bigl[\langle b(x,v), p\rangle + r(x,v)f\bigr]\,,
\end{split}
\end{equation}
and
\begin{equation*}
C^{2}_{\gamma,+}(\Bar{Q}) \;\df\;\{ f \in C^{2}(\Bar{Q}) : \, f\ge0\,,\;
\nabla f \cdot \gamma \;=\; 0 \;\text{on}\; \partial{Q} \}\,.
\end{equation*}

The main result of the paper is the following.

\begin{theorem}\label{T-4.4}
There exists a unique pair
$(\rho,\varphi) \in \mathbb{R}\times C^{2}_{\gamma,+}(\Bar{Q})$
satisfying $\norm{\varphi}_{0;\Bar{Q}}=1$
which solves the pde
\begin{equation*}
\rho\,\varphi(x) \;=\; \cG \varphi(x)
\quad\text{in~} Q\,,
\quad \langle\nabla\varphi,\gamma\rangle = 0 \quad \text{on~} \partial{Q}\,,
\end{equation*}
Moreover,
\begin{align}\label{ET-4.4b}
\rho &\;=\; \inf_{f \in C^{2}_{\gamma,+}(\Bar{Q}),\,f>0}\;
\sup_{\nu \in \PA(\Bar{Q})}\;\int\frac{\cG f}{f}\,d\nu\\
&\;=\; \sup_{f \in C^{2}_{\gamma,+}(\Bar{Q}),\,f>0}\;
\inf_{\nu \in \PA(\Bar{Q})}\;\int\frac{\cG f}{f}\,d\nu\nonumber\,,
\end{align}
where $\PA(\Bar{Q})$  denotes the space of probability measures on
$\Bar{Q}$ with the Prohorov topology.
\end{theorem}

The first part of the theorem is contained in Lemma~\ref{L-4.1}.
The second part is proved in Section~\ref{S4.2}.

The notation used in the paper is summarized below.

\begin{notation}
The standard Euclidean norm in $\RR^{d}$ is denoted by $\abs{\,\cdot\,}$.
The set of nonnegative real numbers is denoted by $\RR_{+}$ and
$\NN$ stands for the set of natural numbers.
The closure, the boundary and the complement
of a set $A\subset\RR^{d}$ are denoted
by $\overline{A}$, $\partial{A}$ and $A^{c}$, respectively.

We adopt the notation
$\partial_{t}\df\tfrac{\partial}{\partial{t}}$, and for $i,j\in\NN$,
$\partial_{i}\df\tfrac{\partial~}{\partial{x}_{i}}$ and
$\partial_{ij}\df\tfrac{\partial^{2}~}{\partial{x}_{i}\partial{x}_{j}}$.
For a nonnegative multi-index $\alpha=(\alpha_{1},\dotsc,\alpha_{d})$
we let $D^{\alpha}\df \partial_{1}^{\alpha_{1}}\dotsb\partial_{d}^{\alpha_{d}}$
and $\abs{\alpha}\df \alpha_{1}+\dotsb+\alpha_{d}$.
For a domain $Q$ in $\RR^{d}$ and $k=0,1,2,\dotsc$,
we denote by $C^{k}(Q)$ the set of functions
$f:Q\to\RR$ whose derivatives $D^{\alpha}f$ for $\abs{\alpha}\le k$
are continuous and bounded.
For $k=0,1,2,\dotsc$, we define
\begin{equation*}
[f]_{k;Q}\;\df\;\max_{\abs{\alpha}=k}\; \sup_{Q}\;\abs{D^{\alpha}f}\quad\text{and}
\quad\norm{f}_{k;Q}\;\df\;\sum_{j=0}^{k}\;[f]_{j;Q}\,.
\end{equation*}
Also for $\delta\in(0,1)$ we define
\begin{equation*}
[g]_{\delta;Q} \;\df\;
\sup_{\substack{x,y \in Q\\x \neq y}}\;
\frac{|g(x) - g(y)|}{|x - y|^{\delta}}\quad\text{and}\quad
\norm{f}_{k+\delta;Q}
\;\df\; \norm{f}_{k;Q} + \max_{\abs{\alpha}=k}\;[D^{\alpha}f]_{\delta;Q}\,.
\end{equation*}
For $k=0,1,2,\dotsc$, and $\delta\in(0,1)$ we denote by
$C^{k+\delta}(Q)$ the space of all real-valued functions $f$ defined on $Q$
such that $\norm{f}_{k+\delta;Q}<\infty$.
Unless indicated otherwise, we always view
$C^{k+\delta}(Q)$ and $C^{k}(Q)$ as topological spaces under the norms
$\norm{\,\cdot\,}_{k+\delta;Q}$ and $\norm{\,\cdot\,}_{k;Q}$ respectively.
We also write $C^{k+\delta}(\Bar{Q})$ and $C^{k}(\Bar{Q})$ if the derivatives
up to order $k$ are continuous on $\Bar{Q}$.
Thus $C^{\delta}(\Bar{Q})$ stands for the Banach space of real-valued functions
defined on $\Bar{Q}$ that are H\"older continuous with exponent $\delta\in(0,1)$.

Let $G$ be a domain in $\RR_{+}\times\RR^{d}$.
Recall that $C^{1,2}(G)$ stands for the set of bounded
continuous real-valued functions $\varphi(t,x)$ defined on $G$ such
that the derivatives $D^{\alpha}\varphi$, $\abs{\alpha}\le 2$
and $\partial_{t}\varphi$ are bounded and continuous in $G$.
Let $\delta\in(0,1)$.
We define
\begin{align*}
[\varphi]_{\nicefrac{\delta}{2},\delta;G} &\;\df\;
\sup_{\substack{(t,x)\ne(s,y)\\(t,x),\,(s,y)\in G}}\;
\frac{\abs{\varphi(t,x)-\varphi(s,y)}}
{\abs{x-y}^{\delta}+\abs{t-s}^{\nicefrac{\delta}{2}}}\,,\\[5pt]
\norm{\varphi}_{\nicefrac{\delta}{2},\delta;G} &\;\df\;
\norm{\varphi}_{0;G} + [\varphi]_{\nicefrac{\delta}{2},\delta;G}\,.
\end{align*}
By $C^{\nicefrac{\delta}{2},\delta}(G)$ we denote the space of functions
$\varphi$ such that $\norm{\varphi}_{\nicefrac{\delta}{2},\delta;G}<\infty$.
The parabolic H\"older space $C^{1+\nicefrac{\delta}{2},2+\delta}(G)$
is the set of all real-valued functions defined on $G$ for which
\begin{equation*}
\norm{\varphi}_{1+\nicefrac{\delta}{2},2+\delta;G} \;\df\;
\max_{\abs{\alpha}\le 2}\;\norm{D^{\alpha}\varphi}_{\nicefrac{\delta}{2},\delta;G}
+\norm{\partial_{t}\varphi}_{\nicefrac{\delta}{2},\delta;G}
\end{equation*}
is finite.
It is well known that $C^{1+\nicefrac{\delta}{2},2+\delta}(G)$ equipped with
the norm $\norm{\varphi}_{1+\nicefrac{\delta}{2},2+\delta;G}$ is a Banach space.

For a Banach space $\mathcal{Y}$ of continuous functions on $\Bar{Q}$
we denote by $\mathcal{Y}_{+}$ its positive cone and by
$\mathcal{Y}_{\gamma}$ the subspace of $\mathcal{Y}$
consisting of the functions $f$ satisfying
$\nabla f \cdot \gamma \;=\; 0$ on $\partial{Q}$.
Also let $\mathcal{Y}^{*}$ denote the dual of $\mathcal{Y}$
and $\mathcal{Y}^{*}_{+}$ the dual cone of $\mathcal{Y}_{+}$.
For example $\bigl(C^{2}_{\gamma}(\Bar{Q})\bigr)^{*}_{+}$ is defined by
\begin{equation*}
\bigl(C^{2}_{\gamma}(\Bar{Q})\bigr)^{*}_{+} \;\df\;
\Bigl\{ \Lambda\in \bigl(C^{2}_{\gamma}(\Bar{Q})\bigr)^{*} :
\Lambda(f)\ge 0\quad\forall f\in C^{2}_{\gamma,+}(\Bar{Q})\Bigr\}\,.
\end{equation*}
We define the operator $\LA_{v}$ on $C^{2}(\Bar{Q})$ by
\begin{equation}\label{E-Lu}
\LA_{v}f(\cdot) \;\df\; \frac{1}{2}\trace\left(a(\cdot)\nabla^{2}f(\cdot)\right)
+ \langle b(\cdot, v), \nabla f(\cdot)\rangle\,, \quad v\in\cV\,,
\end{equation}
where $\nabla^{2}$ denotes the Hessian.
\end{notation}

\section{The Nisio semigroup}

Associated with the above control problem, define for each $t \ge 0$ the operator
$S_{t} : C(\Bar{Q}) \to C(\Bar{Q})$ by
\begin{equation}\label{semi}
S_{t}f(x)\;\df\;\inf_{v(\cdot)}\;
E_{x}\left[e^{\int_{0}^t r(X(s), v(s))ds}f(X(t))\right]\,,
\end{equation}
where the `$\inf$' is over all admissible controls.

A standard consequence of the dynamic programming principle is that this defines
a semigroup, the so called Nisio semigroup.
In fact, the following well known properties thereof can be proved along
the lines of \cite[Theorem~1, pp.~298--299]{Nisi}.
Let
\begin{equation}\label{E-Tu}
T_{t}^uf\;\df\;E_{x}\left[e^{\int_{0}^tr(X^u(s), u)ds}f(X^u(t))\right]\,,
\end{equation}
where $X^u(\cdot)$ is the reflected diffusion in \eqref{sde}
for $v(\cdot) \equiv u\in\cV$.

\begin{theorem}\label{T-3.1}
$\{S_{t}, t \ge 0\}$ satisfies the following properties:
\begin{enumerate}
\item Boundedness: $\|S_{t}f\|_{0;\Bar{Q}} \le
e^{r_{\rm{max}}t}\|f\|_{0;\Bar{Q}}$.
Furthermore, $S_{t}\bm{1} \ge e^{r_{\rm{min}}t} \bm{1}$,
where $\bm{1}$ is the constant function $\equiv 1$, and
$r_{\rm{min}} = \min_{(x, u)}\,r(x,u)$.
\smallskip
\item Semigroup property: $S_{0} = I$
and $S_{t}\circ S_{s} = S_{t + s}$ for $s, t \ge 0$.
\smallskip
\item Monotonicity:
$f \ge$ (resp., $>$) $g \;\Longrightarrow\; S_{t}f \ge$ (resp., $>$) $S_{t}g$.
\smallskip
\item Lipschitz property:
$\|S_{t}f - S_{t} g\|_{0;\Bar{Q}} \le e^{r_{\rm{max}}t}\|f - g\|_{0;\Bar{Q}}$.
\smallskip
\item Strong continuity: $\|S_{t}f - S_{s}f\|_{0;\Bar{Q}} \to 0$ as $t \to s$.
\smallskip
\item Envelope property: $T^u_{t}f \ge S_{t}f$ for all $u \in U$, and
$S_{t}f \ge S_{t}'f$ for any
other $\{S_{t}'\}$ satisfying  this along with the foregoing properties.
\smallskip
\item Generator: the infinitesimal generator of $\{S_{t}\}$ is $\cG$
defined in \eqref{E-gen}.

\end{enumerate}
\end{theorem}

We can say more by invoking p.d.e.\ theory.
We start with the following theorem that characterizes $S_{t}$ as the
solution of a parabolic p.d.e.

\begin{theorem}\label{T-3.2}
For each $f \in C^{2+\delta}_{\gamma}(\Bar{Q})$,
$\delta\in(0,\beta_{0})$, and $T>0$,
the quasi-linear parabolic p.d.e.
\begin{equation}\label{pde}
\frac{\partial}{\partial t}\psi(t,x) \;=\; \inf_{v\in\cV}\;
\bigl(\LA_{v} \psi(t,x) + r(x,v)\psi(t,x)\bigr)
\quad \text{~in~} (0,T]\times Q\,,
\end{equation}
with $\psi(0,x) = f(x)$ for all $x \in \Bar{Q}$
and
\begin{equation*}
\langle\nabla \psi(t,x), \gamma(x)\rangle = 0\qquad \text{for all}~
(t,x) \in (0,T]\times\partial{Q}\,,
\end{equation*} has a unique solution in
$C^{1+\nicefrac{\delta}{2},2+\delta}\bigl([0,T]\times\Bar{Q}\bigr)$.
The solution $\psi$ has the stochastic representation
\begin{equation}\label{stochasticrep}
\psi(t,x) \;=\; \inf_{v(\cdot)}\;E_{x} \left[e^{\int^t_{0} r(X(s), v(s))\,ds}
 f(X(t))\right] \qquad \forall (t,x)\in[0, T]\times\Bar{Q} \,.
\end{equation}
Moreover,
\begin{align*}
\|\psi\|_{1,2;[0,T]\times\Bar{Q}} &\;\le\; K_{1}, \\[5pt]
\|\nabla^{2}\psi(s, \cdot)\|_{\delta;Q} &\;\le\;
K_{2} \quad\text{for all~} s\in [0,T]\,,
\end{align*}
where the constants $K_{1},\,K_{2} > 0$ depend only on
$T, \|a\|_{1+\beta_{0}; Q}$, the Lipschitz constants of $b, r$,
the lower bound on the  eigenvalues of $a$, the boundary
$\partial{Q}$ and $\|f\|_{2+\delta;Q}$.
\end{theorem}

\begin{proof}
This follows by \cite[Theorem~7.4, p.~491]{Lady} and
\cite[Theorem~7.2, pp.~486--487]{Lady}.
\end{proof}

\begin{lemma}\label{L-3.1}
Let $\delta\in(0, \beta_{0})$.
For each $t > 0$, the map $S_{t} : C^{2+\delta}_{\gamma}(\Bar{Q})
\to C^{2+\delta}_{\gamma}(\Bar{Q})$ is compact.
\end{lemma}

\begin{proof}
Suppose $f\in C^{2+\delta}_{\gamma}(\Bar{Q})$ for some $\delta\in(0,\beta_{0})$.
Fix any $T>0$.
Let $g:[0,\infty)\to[0,\infty)$ be a smooth function such that
$g(0)=0$ and $g(s)=1$ for $s\in[\nicefrac{T}{2},\infty)$.
Define $\Tilde{\psi}(t,x) = g(t)\psi(t,x)$, with $\psi$ as
in Theorem~\ref{T-3.2}.
Then $\Tilde{\psi}$ satisfies
\begin{equation}\label{E-cutoff}
\frac{\partial}{\partial t}\Tilde{\psi}(t,x)
- \frac{1}{2}\trace\left(a(x)\nabla^{2}\Tilde{\psi}(t,x)\right)\;=\;
\frac{\partial g}{\partial t}(t) \psi(t,x)
+ g(t)\cH\bigl(x,\psi(t,x),\nabla \psi(t,x)\bigr)
\end{equation}
in $(0,\infty)\times Q$, $\Tilde{\psi}(0,x) = 0$ on $\Bar{Q}$ and
 $\langle\nabla \Tilde{\psi}(t,x), \gamma(x)\rangle = 0$ for all
$(t,x) \in (0,\infty)\times\partial{Q}$.
It is well known that  $\nicefrac{\partial}{\partial_{x_{i}}}$ is a
bounded operator
from $C^{1+\nicefrac{\delta}{2},2+\delta}\bigl([0,T]\times\Bar{Q}\bigr)$
to $C^{\nicefrac{(1+\delta)}{2},1+\delta}\bigl([0,T]\times\Bar{Q}\bigr)$
\cite[p.~126]{Krylov-Lect}.
In particular
\begin{equation*}
\sup_{x\in\Bar{Q}}\;
\sup_{s\ne t}\;
\frac{\bigl| \partial_{x^{i}}\psi (s,x) - \partial_{x^{i}}\psi(t,x)\bigr|}
{|s-t|^{\nicefrac{(1+\delta)}{2}}}\;<\;\infty\,.
\end{equation*}
Since $\cH$ is Lipschitz in its arguments and $g$ is smooth it follows
that the r.h.s.\ of \eqref{E-cutoff} is
in $C^{\nicefrac{\beta}{2},\beta}\bigl([0,T]\times\Bar{Q}\bigr)$
for any $\beta\in(0,1)$.
Then it follows by the interior estimates in
\cite[Theorem~10.1, pp.~351-352]{Lady} that
$\Tilde{\psi}\in C^{1+\nicefrac{\beta}{2},2+\beta}\bigl([T,T+1]\times\Bar{Q}\bigr)$
for all $\beta\in(0,\beta_{0})$.
Since $\psi=\Tilde{\psi}$ on $[T,T+1]$ it follows that
$S_{T}f\in C^{2+\beta}_{\gamma}(\Bar{Q})$ for all $\beta\in(0,\beta_{0})$.
Since the inclusion $C^{2+\beta}_{\gamma}(\Bar{Q})\hookrightarrow
C^{2+\delta}_{\gamma}(\Bar{Q})$ is compact for $\beta>\delta$, the result follows.
\end{proof}

\section{An abstract Collatz--Wielandt formula}

The classical Collatz--Wielandt formula (see \cite{Coll,Weil})
characterizes the principal (i.e., the Perron-Frobenius) eigenvalue $\kappa$
of an irreducible non-negative matrix $Q$ as
(see \cite[Chapter~8]{Mey})
\begin{align*}
\kappa\;&=\;\max_{\{x = (x_{1},\dotsc,x_d) \,:\, x_i \ge 0\}}\;
\min_{\{i \,:\, x_i > 0\}}
\left(\frac{(Qx)_i}{x_i}\right)\\[5pt]
&=\; \min_{\{x = (x_{1}, \dotsc, x_d) \,:\, x_i > 0\}}\;
\max_{\{i \,:\, x_i > 0\}}\left(\frac{(Qx)_i}{x_i}\right).
\end{align*}
An infinite dimensional version of this was recently given by Chang \cite{Chan}
as follows.
Let $\cX$ be a real Banach space
with order cone $P$, i.e., a nontrivial closed subset of $\cX$.
Define $-P \df \{-x : x\in P\}$ and $\dot{P}\;\df\;P\backslash\{\theta\}$.
We assume that the cone $P$ satisfies
\begin{enumerate}
\item[(a)]
$tP\subset P$ for all $t\ge0$, where $tP = \{tx : x\in P\}\,$;
\item[(b)]
$P+P\subset P\,$;
\item[(c)]
$P\cap(-P)=\{\theta\}$, where $\theta$ denotes the zero vector of $\cX$.
\end{enumerate}
We write $x \preceq y$ if $y - x \in P$, and $x \prec y$ if
$x \preceq y$ and $x\ne y$.
Define the dual cone
\begin{equation*}
P^*\;\df\;\{x \in \cX^* : \langle x^*,x\rangle\ge 0\quad\forall x \in P\}\,.
\end{equation*}
A map $T : \cX \to \cX$ is said to be \emph{increasing} if
$x \preceq y \Longrightarrow T(x) \preceq T(y)$, and \emph{strictly increasing}
if $x \prec y \Longrightarrow T(x) \prec T(y)$.
If $\interior(P) \neq \varnothing$, and $T : \dot{P} \to \interior(P)$,
then $T$ is called \emph{strongly positive}, and if
$x \prec y \Longrightarrow T(y) - T(x)\in\interior(P)$
it is called \emph{strongly increasing}.
It is called \emph{positively 1-homogeneous} if $T(tx) = tT(x)$ for all $t > 0$
and  $x \in \cX$.
Also, a map $T : \cX \to \cX$ is called
\emph{completely continuous} if it is continuous and compact.
A generalization of the Kre{\u\i}n-Rutman theorem appears in
\cite{Maha}.
However the hypotheses in \cite[Theorem~2]{Maha} are not sufficient
for uniqueness of an eigenvector in $P$, so the conclusions of that theorem
are not correct.
The same error has propagated in \cite[Theorems~1.4, 4.8, and 4.13]{Chan}.
For a detailed discussion on this see the forthcoming paper \cite{Ari}.
A corrected version of \cite[Theorem~2]{Maha} is as follows:

\begin{theorem}\label{T-4.1}
Let $T : \cX \to \cX$ be an increasing,
positively 1-homogeneous, completely continuous
map such that for some $u \in P$ and $M > 0$, $MT(u) \succeq u$.
Then there exist $\lambda > 0$ and $ \Hat{x} \in \dot{P}$ such that
$T(\Hat{x}) = \lambda \Hat{x}$.
Moreover, if $T$ is strongly increasing
then $\lambda$ is the unique eigenvalue with an eigenvector in $P$.
\end{theorem}


The following is proved in \cite{Chan}:

\begin{theorem}\label{T-4.2}
Let $T$ and $\lambda$ be as in the preceding theorem.
Define:
\begin{align*}
P^*(x) &\;\df\; \{x^* \in P^* : \langle x^*, x \rangle > 0\}\,, \\
r_*(T) &\;\df\; \sup_{x \in \dot{P}}\;\inf_{x^* \in P^*(x)}\;
\frac{\langle x^*, T(x)\rangle}{\langle x^*, x\rangle}\,, \\
r^*(T) &\;\df\; \inf_{x \in \dot{P}}\; \sup_{x^* \in P^*(x)}\;
\frac{\langle x^*, T(x)\rangle}{\langle x^*, x\rangle}\,.
\end{align*}
If $T$ is strongly increasing
then $\lambda = r^*(T) = r_*(T)$.
\end{theorem}

Uniqueness of the positive eigenvector can be obtained under additional
assumptions.
In this paper we are concerned with \emph{superadditive} operators $T$, in
other words operators $T$ which satisfy
\begin{equation*}
T(x+y)\;\succeq\; T(x) + T(y)\qquad \forall x,y\in\cX\,.
\end{equation*}
We have the following simple assertion:

\begin{corollary}\label{C-4.1}
Let $T : \cX \to \cX$ be a superadditive,
positively 1-homogeneous, strongly positive, completely continuous map.
Then there exists a unique
$\Hat{x} \in \dot{P}$ with $\norm{\Hat{x}}=1$, where $\norm{\,\cdot\,}$ denotes
the norm in $\cX$, such that
$T(\Hat{x}) = \lambda \Hat{x}$, with $\lambda > 0$.
\end{corollary}

\begin{proof}
It is clear that strong positivity implies that for any
$x\in\cX$ there exists $M>0$ such that $MT(x)\succeq x$.
By superadditivity $T(x-y)\preceq T(x) - T(y)$.
Hence if $x\succ y$, by strong positivity we obtain
$T(x) - T(y)\in\interior(P)$.
Therefore every superadditive, strongly positive map is strongly increasing.
Existence of a unique eigenvalue with an eigenvector in $P$ then
follows by Theorem~\ref{T-4.1}.
Suppose $\Hat{x}$ and $\Hat{y}$ are two distinct unit eigenvectors in $P$.
Since, by strong positivity $\Hat{x}$ and $\Hat{y}$ are in $\interior(P)$
there exists $\alpha>0$ such that
$\Hat{x}-\alpha\Hat{y}\in \dot{P}\setminus\interior(P)$.
Since $T$ is strongly increasing we obtain
\begin{equation*}
\lambda(\Hat{x}-\alpha\Hat{y}) \;=\;
T(\Hat{x})-T(\alpha\Hat{y})
\succeq T(\Hat{x}-\alpha\Hat{y})\in\interior(P)\,,
\end{equation*}
a contradiction.
Uniqueness of a unit eigenvector in $P$ follows.
\end{proof}

An application of Theorem~\ref{T-4.1} and Corollary~\ref{C-4.1} provides us
with the following result for strongly continuous semigroups of operators.

\begin{corollary}\label{C-4.2}
Let $\cX$ be a Banach space with order cone $P$
having non-empty interior.
Let $\{S_{t},\, t\ge0\}$ be a strongly continuous semigroup of
superadditive, strongly positive, positively 1-homogeneous,
completely continuous operators on $\cX$.
Then there exists a unique $\rho \in \mathbb{R}$ and a unique
$\Hat{x}\in\interior(P)$, with $\norm{\Hat{x}}=1$, such that
$S_{t}\Hat{x} = e^{\rho t}\Hat{x}$ for all $t\ge0$.
\end{corollary}

\begin{proof}
By Theorem~\ref{T-4.1} and Corollary~\ref{C-4.1} there exists a unique
$\lambda(t) > 0$ and a unique $x_{t}\in P$ satisfying $\norm{x_{t}}= 1$, such that
$S_{t} x_{t} = \lambda(t) x_{t}$.
By the uniqueness of a unit eigenvector in $P$ and the
semigroup property it follows that there exists $\Hat{x}\in\cX$ such that
$x_{t}=\Hat{x}$ for all dyadic rational numbers $t>0$.
On the other hand, from the strong continuity it follows that if a sequence of
dyadic rationals $t_n \ge 0$, $n \ge 1$,  converges to some $t > 0$,
then $\lambda(t_n)$ is a Cauchy sequence
and its limit point $\lambda'$ is an eigenvalue of $S_{t}$
corresponding to the eigenvector $\Hat{x}$ and therefore
$\lambda(t)=\lambda'$ and $x_{t} = \Hat{x}$ by the uniqueness thereof.
Strong continuity then implies
that $\lambda(\cdot)$ is continuous and by the semigroup property
and positive 1-homogeneity we have $\lambda(t + s) = \lambda(t)\lambda(s)$
for all for $t, s > 0$.
It follows that $\lambda(t) = e^{\rho t}$ for some $\rho \in \mathbb{R}$.
\end{proof}

\subsection{Stability}
Concerning the time-asymptotic behavior of $S_{t}x$ we have the following.

\begin{theorem}\label{T-general}
Let $\cX$, $\{S_{t}\}$, $\rho$ and $\Hat{x}$
be as in Corollary~\ref{C-4.2}.
Then
\begin{itemize}
\item[(i)]
The set
\begin{equation*}
\cO_{1}\;\df\;\bigl\{e^{-\rho t} S_{t}x : x\in P\,,\; \norm{x}\le1\,,\;t\ge1\bigr\}
\end{equation*}
is relatively compact in $\cX$.
\smallskip
\item[(ii)]
There exists $\alpha^{*}(x)\in\RR_{+}$ such that
\begin{equation*}
\lim_{t\to\infty}\;
\bnorm{e^{-\rho t}S_{t} x - \alpha^{*}(x)\, \Hat{x}}
\;\xrightarrow[t\to\infty]{}\;0\quad \forall x\in\Dot{P}\,.
\end{equation*}
\item[(iii)]
Suppose that additionally the following properties hold:
\begin{itemize}
\item[(P1)]
For every $M>0$ there exist $\tau\in(0,1)$
and a positive constant $\zeta_{0}=\zeta_{0}(M)$
such that
\begin{equation*}
\norm{S_{\tau}(\Hat{x} - z)} + \norm{S_{\tau}z} \;\ge\; \zeta_{0}
\end{equation*}
for all $z\in P$ such that $z\preceq \Hat{x}$ and $\norm{z}\le M$.
\item[(P2)]
For every compact set $\cK\subset P$ there exists a constant
$\zeta_{1}=\zeta_{1}(\cK)$
such that
$x\in \cK$ and $x\preceq \alpha\, \Hat{x}$ imply $\norm{x}\le \alpha\,\zeta_{1}$.
\end{itemize}
Then the convergence is exponential:
there exists $M_{0}>0$ and $\theta_{0}>0$ such that
\begin{equation*}
\norm{e^{-\rho t}S_{t} x - \alpha^{*}(x)\, \Hat{x}}\;\le\;
M_{0} e^{-\theta_{0} t}\,\norm{x}\qquad\text{for all~} t\ge0
\quad\text{and all~}x\in\Dot{P}\,.
\end{equation*}
\end{itemize}
\end{theorem}

\begin{proof}
Without loss of generality we can assume $\rho=0$.
For $t\ge0$ and $x\in P$ we define
\begin{align*}
\ua(x) &\;\df\; \sup\;\{a\in\RR : x- a\, \Hat{x} \in P\}\\[5pt]
\oa(x) &\;\df\; \inf\;\{a\in\RR : a\, \Hat{x} - x  \in P\}\,.
\end{align*}
Since $\Hat{x}\in\interior(P)$ it follows  that $\ua(x)$
and $\oa(x)$ are finite and $\oa(x)\ge\ua(x)\ge0$.
Note also that for $x\in\Dot{P}$ we have
$\oa(x) >0$ and since $S_{t}x\in\interior(P)$ we have
$\ua(S_{t}x)>0$ for all $t>0$.
It is also evident from the definition that 
\begin{equation*}
\ua(\lambda x)\;=\;\lambda\,\ua(x)\quad\text{and}\quad
\oa(\lambda x)\;=\;\lambda\,\oa(x)\quad
\text{for all}~x\in\Dot{P}\,,~\lambda\in\RR_{+}\,.
\end{equation*}

By the increasing property and the positive $1$-homogeneity
of $S_{t}$ we obtain
$S_{t+s}x - \ua(S_{s}x)\,\Hat{x}\in P$ for all $x\in P$ and $t\ge0$
and this implies that
 $\ua(S_{t+s}x)\ge\ua(S_{s}x)$ for all $t\ge 0$ and $x\in P$.
It follows that for any $x\in P$ the map $t\mapsto \ua(S_{t}x)$ is non-decreasing.
Similarly, the map $t\mapsto \oa(S_{t}x)$ is non-increasing.

We next show that the orbit $\cO$ of the unit ball in $P$ defined by
\begin{equation*}
\cO\;\df\;\bigl\{S_{t}x : x\in P\,,\; \norm{x}\le1\,,\;t\ge0\bigr\}
\end{equation*}
is bounded.
Suppose not.
Then we can select a sequence
$\{x_{n}\}\subset\Dot{P}$
with $\norm{x_{n}}=1$, and an increasing sequence $\{t_{n}\,,\;n\in\NN\}$
such that $\norm{S_{t_{n}} x_{n}}\to \infty$ as $n\to\infty$ and such that
$\norm{S_{t_{n}} x_{n}}\ge\norm{S_{t} x_{n}}$ for all $t\le t_{n}$.
By the properties of the sequence $\{S_{t_{n}}\}$ the sequence
$\left\{\frac{S_{t_{n}-2}\, x_{n}}{\norm{S_{t_{n}}x_{n}}}\right\}$ is bounded
and this implies  that
$\left\{\frac{S_{t_{n}-1}\, x_{n}}{\norm{S_{t_{n}}x_{n}}}\right\}$ is relatively
compact.
Let $y\in\cX$ be any limit point of
$\frac{S_{t_{n}-1}\, x_{n}}{\norm{S_{t_{n}}x_{n}}}$
as $n\to\infty$.
By continuity of $S_{1}$ it follows that
$\norm{S_{t_{n}}\,x_{n}}\le k_{1}\norm{S_{t_{n}-1}\, x_{n}}$ for some $k_{1}>0$.
This implies that $\norm{y}\ge k_{1}^{-1}$.
Therefore $y\in\Dot{P}$ which in turn implies that $\ua(S_{1}y)>0$.
It is straightforward to show that the map $x\mapsto\ua(x)$ is continuous.
Therefore, we have
\begin{equation}\label{e-bounded1}
\ua\left(\frac{S_{t_{n}}x_{n}}{\norm{S_{t_{n}}x_{n}}}\right)
\;=\;
\ua\left(S_{1}\left(\frac{S_{t_{n}-1}x_{n}}{\norm{S_{t_{n}}x_{n}}}\right)\right)
\;\xrightarrow[n\to\infty]{}\;\ua(S_{1}y)\,.
\end{equation}
On the other hand, it holds that
\begin{equation}\label{e-bounded2}
\ua(S_{t_{n}}x_{n}) \;=\; \norm{S_{t_{n}}x_{n}}\;
\ua\left(\frac{S_{t_{n}}x_{n}}{\norm{S_{t_{n}}x_{n}}}\right)\,.
\end{equation}
Since $\Hat{x}\in\interior(P)$ the constant $\kappa_{1}$ defined by
\begin{equation}\label{E-kappa1}
\kappa_{1}\;\df\;\sup_{x\in P,\,\norm{x}=1}\;\oa(x)
\end{equation}
is finite.
Since $\ua(S_{1}y)>0$ and $\norm{S_{t_{n}}x_{n}}$ diverges,
\eqref{e-bounded1}--\eqref{e-bounded2} imply that $\ua(S_{t_{n}}x_{n})$
diverges which
is impossible since
\begin{equation*}
\ua(S_{t_{n}}x_{n})\;\le\;\oa(S_{t_{n}}x_{n})\;\le\;\oa(x_{n})\;\le\;\kappa_{1}\,.
\end{equation*}

Since $\cO$ is bounded in $\cX$, there exists a constant $k_{0}$
such that
\begin{equation}\label{E-bark}
\norm{S_{t} x} \;\le\; k_{0} \norm{x}
\qquad \forall t\in[0,1]\,,\quad\forall x\in P\,.
\end{equation}

That the set $\cO_{1}$
is relatively compact for each $x\in\cX$ now easily follows.
Indeed, since $\cO(x)$ is bounded, by the semigroup property we obtain
\begin{equation*}
\cO_{1} \;=\;
\bigl\{S_{1}(S_{t-1}x) : x\in P\,,\; \norm{x}=1\,,\;t\ge1\bigr\}
\;\subset\; S_{1}\bigl(\cO\bigr)\,,
\end{equation*}
and the claim follows since by hypothesis $S_{1}$ is a compact map.

For all $t\ge s\ge 0$ we have
\begin{align}
S_{t}\bigl(S_{s}x -\ua(S_{s}x)\,\Hat{x}\bigr)
&\;\preceq\; S_{t+s}x - \ua(S_{s}x)\,\Hat{x}\,,\label{e-bound1a}\\[5pt]
S_{t}\bigl(\oa(S_{s}x)\,\Hat{x}-S_{s}x\bigr)
&\; \preceq\; \oa(S_{s}x)\,\Hat{x} -S_{t+s}x\,.\label{e-bound1b}
\end{align}
Let $s=t_{n}$ in \eqref{e-bound1a} and take limits along some converging sequence
$S_{t_{n}}x\to \Bar{x}$ as $n\to\infty$, for some $\Bar{x}\in P$, to obtain
\begin{equation}\label{e-ineqb}
\ua^{*}(x) \Hat{x} + S_{t}\bigl(\Bar{x} -\ua^{*}(x) \Hat{x}\bigr)
\;\preceq\; S_{t}\Bar{x}\,,
\end{equation}
where $\ua^{*}(x)\;\df\;\lim_{t\uparrow\infty}\; \ua(S_{t} x)$.
Since $\Bar{x}$ is an $\omega$-limit point of $S_{t}x$ it follows
that $\ua(S_{t}\Bar{x})=\ua^{*}(x)$ for all $t\ge0$.
Therefore $S_{t}\Bar{x}-\ua^{*}(x) \Hat{x}\notin \interior(P)$
for all $t\ge0$, which implies by \eqref{e-ineqb} and the strong positivity of
$S_{t}$  that $\Bar{x} -\ua^{*}(x) \Hat{x}=0$.
A similar argument shows that $\Bar{x} = \oa^*(x) \Hat{x}$,
where $\oa^{*}(x)\;\df\;\lim_{t\uparrow\infty}\;\oa(S_{t}x)$.
We let $\alpha^{*}\df\oa^{*}=\ua^{*}$.

It remains to prove that convergence is exponential.
Since the orbit $\cO$ is bounded and $\Hat{x}\in\interior(P)$
it follows that the set $\{\oa(S_{t}x) : t\ge 0\,,~x\in P\,,~\norm{x}\le 1\}$
is bounded.
Therefore since the orbit $\cO_{1}$ is also relatively compact, it follows that
the set
\begin{equation*}
\cK_{1}\;\df\;
\bigl\{ S_{k} x -\ua(S_{k}x)\Hat{x}\,,\;
\oa(S_{k}x)\Hat{x}-S_{k}x : k\ge 1\,,~x\in P\,,~\norm{x}\le 1\bigr\}
\end{equation*}
is a relatively compact subset of $P$.
Define
\begin{equation*}
\eta(S_{k}x)\;\df\;\oa(S_{k}x) - \ua(S_{k}x)\,,\quad k=1,2,\dotsc
\end{equation*}
By property (P2), since
\begin{align*}
S_{k} x -\ua(S_{k}x)\Hat{x} &\;\preceq\;
\eta(S_{k}x)\,\Hat{x}\,,\\[5pt]
\oa(S_{k}x)\Hat{x}-S_{k}x &\;\preceq\;
\eta(S_{k}x)\,\Hat{x}\,,
\end{align*}
it follows that for some $\zeta_{1}=\zeta_{1}(\cK_{1})$ we have
\begin{equation}\label{E-b00}
\max\;\bigl\{\norm{S_{k} x -\ua(S_{k}x)\Hat{x}}\,,\;
\norm{\oa(S_{k}x)\Hat{x}-S_{k}x}\bigr\}
\;\le\; \zeta_{1}\,\eta(S_{k}x)
\end{equation}
for all $k\ge1$ and $x\in P$ with $\norm{x}\le 1$.
Define
\begin{equation*}
\ucZ_{k}(x)\;\df\;\frac{S_{k} x -\ua(S_{k}x)\Hat{x}}{\eta(S_{k}x)}\,,\qquad
\ocZ_{k}(x)\;\df\;\frac{\oa(S_{k}x)\Hat{x}-S_{k}x}{\eta(S_{k}x)}\,,
\end{equation*}
provided $\eta(S_{k}x)\ne0$, which is equivalent to $S_{k}x\ne\Hat{x}$.
By \eqref{E-b00} the set
\begin{equation*}
\Tilde\cK_{1}\;\df\;
\bigl\{ \ucZ_{k}(x)\,,\;\ocZ_{k}(x)
: k\ge 1\,,~x\in \Dot{P}\setminus\{\Hat{x}\}\,,~\norm{x}\le 1\bigr\}
\end{equation*}
lies in the ball of radius $\zeta_{1}$ centered at the origin of $\cX$.
Therefore, since $\ocZ_{k}(x)=\Hat{x}-\ucZ_{k}(x)$,
by property (P1) there exists
$\zeta_{0}=\zeta_{0}(\zeta_{1})>0$ and $\tau\in(0,1)$ such that
\begin{equation}\label{E-b000}
\norm{S_{\tau}\ucZ_{k}(x)}+\norm{S_{\tau}\ocZ_{k}(x)}\ge \zeta_{0}\qquad
\forall k=1,2,\dotsc,\quad
\forall x\in\Dot{P}\setminus\{\Hat{x}\}\,,~\norm{x}\le 1
\end{equation}
Let
\begin{equation*}
A_{k}(x) \;\df\; \sup\;\bigl\{\alpha\in\RR:
\{S_{1}\ucZ_{k}(x)-\alpha\,\Hat{x}\}\cup
\{S_{1}\ocZ_{k}(x)-\alpha\,\Hat{x}\}\subset P\bigr\}\,.
\end{equation*}
We claim that
\begin{equation}\label{E-b01}
\zeta_{2}\;\df\; \inf\; \bigl\{A_{k}(x)
:  k\ge 1\,,~x\in \Dot{P}\setminus\{\Hat{x}\}\,,~\norm{x}\le 1\bigr\}\;>\;0\,.
\end{equation}
Indeed if the claim is not true then by \eqref{E-b000}
and the definition of $A_{k}$ there exists a sequence
$z_{k}$ taking values in
\begin{equation*}
\bigl\{\ucZ_{k}(x),\ocZ_{k}(x):
x\in \Dot{P}\setminus\{\Hat{x}\}\,,~\norm{x}\le 1\bigr\}
\end{equation*}
for each $k=1,2,\dotsc$,
such that $\norm{S_{\tau}z_{k}}\ge \nicefrac{\zeta_{0}}{2}$
and such that $\ua{(S_{1}z_{k})}\to 0$ as $k\to\infty$.
However, since $\Tilde\cK_{1}$ is bounded, it follows
that  $S_{\tau}\bigl(\Tilde\cK_{1}\bigr)$ is a relatively
compact subset of $\interior(P)$.
Therefore the limit set of $S_{\tau}z_{k}$
is nonempty and any limit point $y\in P$ of $S_{\tau}z_{k}$ satisfies
$\norm{y}\ge\nicefrac{\zeta_{0}}{2}$.
Since $\ua{(S_{1}z_{k})} = \ua{(S_{1-\tau}S_{\tau}z_{k})}$
and $z\mapsto \ua{(S_{1-\tau}z)}$ is continuous on $P$,
any such limit point $y$ satisfies
$\ua(S_{1-\tau}y)=0$ which contradicts the strong positivity hypothesis.

Equation \eqref{E-b01} implies that
\begin{equation}\label{E-b02}
\ua\bigl(S_{1}\bigl(\oa(S_{k}x)\,\Hat{x}-S_{k}x\bigr)\bigr)
+ \ua\bigl(S_{1}\bigl(S_{k}x -\ua(S_{k}x)\,\Hat{x}\bigr)\bigr)
\;\ge\;\zeta_{2}\bigl(\oa(S_{k}x)-\ua(S_{k}x)\bigr)
\end{equation}
for all $x\in \Dot{P}\setminus\{\Hat{x}\}$ with $\norm{x}\le 1$, and by
$1$-homogeneity, for all $x\in \Dot{P}\setminus\{\Hat{x}\}$.

By \eqref{e-bound1a}--\eqref{e-bound1b} we have
\begin{equation}\label{E-lower3}
\begin{split}
S_{1}\bigl(S_{k}x -\ua(S_{k}x)\,\Hat{x}\bigr)
\; \preceq\; S_{k+1}x - \ua(S_{k}x)\,\Hat{x}\,,\\[5pt]
S_{1}\bigl(\oa(S_{k}x)\,\Hat{x}-S_{k}x\bigr) \;\preceq\;
\oa(S_{k}x)\,\Hat{x} - S_{k+1}x\,.
\end{split}
\end{equation}
In turn \eqref{E-lower3} implies that
\begin{equation}\label{E-b04}
\begin{split}
\ua(S_{k+1}x)&\;\ge\;\ua(S_{k}x)
+\ua\bigl(S_{1}\bigl(S_{k}x -\ua(S_{k}x)\,\Hat{x}\bigr)\bigr)
\,,\\[5pt]
\oa(S_{k+1}x)&\;\le\; \oa(S_{k}x)
-\ua\bigl(S_{1}\bigl(\oa(S_{k}x)\,\Hat{x}-S_{k}x\bigr)\bigr)\,.
\end{split}
\end{equation}
By \eqref{E-b02} and \eqref{E-b04} we obtain that
\begin{equation*}
\eta(S_{k}x) - \eta(S_{k+1}x)\; \ge\;  \zeta_{2}\,\eta(S_{k}x)\,,
\end{equation*}
which we write as
\begin{equation}\label{E-b05}
\eta(S_{k+1}x) \;\le\;(1-\zeta_{2})\,\eta(S_{k}x)\,,\qquad
\quad k=1,2,\dotsc
\end{equation}
We add the inequalities
\begin{equation*}
\begin{split}
\norm{S_{k}x -\alpha^{*}(x)\,\Hat{x}} &\;\le\;
\norm{S_{k}x -\ua(S_{k}x)\,\Hat{x}} +
\alpha^{*}(x)-\ua(S_{k}x)\,,\\[5pt]
\norm{\alpha^{*}(x)\,\Hat{x}-S_{k}x} &\;\le\;
\norm{\oa(S_{k}x)\,\Hat{x}-S_{k}x} +
\oa(S_{k}x)-\alpha^{*}(x)
\end{split}
\end{equation*}
and use \eqref{E-b00} and \eqref{E-b05} to obtain
\begin{align}\label{E-eta0}
2\,\norm{S_{k}x -\alpha^{*}(x)\,\Hat{x}} &\;\le\;
2\zeta_{1}\,\eta(S_{k}x) + \eta(S_{k}x)\\[5pt]
&\;\le\;(2\zeta_{1}+1)\eta(S_{k}x)\nonumber\\[5pt]
&\;\le\;(2\zeta_{1}+1)(1-\zeta_{2})^{k-1}\,\eta(S_{1}x)\,,\qquad
\quad k=1,2,\dotsc\nonumber
\end{align}
We have
\begin{align}\label{E-eta1}
\eta(S_{1}x) &\;=\; \oa(S_{1} x) - \ua(S_{1} x)\\[3pt]
&\;\le\; \oa(S_{1} x)\nonumber\\[3pt]
&\;\le\; \kappa_{1}\,\norm{S_{1} x}\nonumber\\[3pt]
&\;\le\;\kappa_{1}\,k_{0}\,\norm{x}\,,\nonumber
\end{align}
where $k_{0}$ is the continuity constant in \eqref{E-bark}
and $\kappa_{1}$ is defined in \eqref{E-kappa1}.
Let $\lfloor t\rfloor$ denote the integral part of a number $t\in\RR_{+}$.
We define
\begin{equation*}
M_{0}\;\df\;  \frac{\kappa_{1}\,k_{0}^{2}\,(2\zeta_{1}+1)}
{2}\qquad\text{and}\qquad \theta_{0}\;\df\;
-\log(1-\zeta_{2})\,,
\end{equation*}
and combine \eqref{E-eta0}--\eqref{E-eta1} to obtain
\begin{align*}
\norm{S_{t}x -\alpha^{*}(x)\,\Hat{x}} &\;\le\;
\frac{M_{0}}{k_{0}}(1-\zeta_{2})^{\lfloor t\rfloor-1}
\bnorm{S_{t-\lfloor t\rfloor} x}\\[5pt]
&\;\le\; M_{0} e^{-\theta_{0} t}\,\norm{x}\,.
\end{align*}
The proof is complete.
\end{proof}

\begin{remark}
Recall that the cone $P$ is called \emph{normal} if there exists
a constant $K$ such that 
$\norm{x}\le K \norm{y}$ whenever $0\preceq x\preceq y$.
It might appear that property (P2) 
in Theorem~\ref{T-general} is weaker than normality of the cone.
However it turns out that (P2) together with the fact that $\Hat{x}$ is
an interior point imply that $P$ is normal.
This is shown in Lemma~\ref{L4.1} below.

Also $\tau$ in (P1) in Theorem~\ref{T-general} can be any
positive constant and need not be restricted
to lie in $(0,1)$.
The proof of geometric convergence follows in the same manner,
by using the iterates $S_{k(\tau+1)}$ instead of $S_{k}$.
\end{remark}

\begin{lemma}\label{L4.1}
Consider the following properties:
\begin{itemize}
\item[(P2$^{\prime}$)]
There exists a constant $\zeta_{1}^\prime>0$ such that
$x\in P$ and $x\preceq \Hat{x}$ imply $\norm{x}\le \zeta_{1}^\prime$.
\item[(P2$^{\prime\prime}$)]
$P$ is normal.
\end{itemize}
Then
\textup{(P2)} $\Longleftrightarrow$  \textup{(P2$^{\prime}$)}
$\Longleftrightarrow$  \textup{(P2$^{\prime\prime}$)}
\end{lemma}

\begin{proof}
If (P2$^{\prime}$) doesn't hold then there exists $\{x_{n}\}\subset P$
with $x_{n}\preceq\Hat{x}$ and $\norm{x_{n}}\nearrow\infty$.
Hence $\{\norm{x_{n}}^{-2} x_{n}\}$ is precompact, 
and since $\norm{x_{n}}^{-2} x_{n}\preceq \norm{x_{n}}^{-2}\Hat{x}$
this implies by (P2) that 
$\norm{x_{n}}^{-2} \norm{x_{n}} \le \norm{x_{n}}^{-2} \zeta_{1}$ for some
$\zeta_{1}>0$.
This contradicts $\norm{x_{n}}\nearrow\infty$ and so
(P2) cannot hold.
Therefore (P2) $\Longrightarrow$ (P2$^{\prime}$).
The other direction is obvious.

Since $\Hat{x}\in\interior(P)$, there exists $\varepsilon>0$
such that $\norm{y}\le\varepsilon$ implies that $y\preceq \Hat{x}$.
Suppose $0\preceq x\preceq y$.
By scaling we have
\begin{equation}\label{EA-3}
0\;\preceq\; \frac{\varepsilon}{\norm{y}} x\;\preceq\;
\frac{\varepsilon}{\norm{y}} y \;\preceq\; \Hat{x}\,.
\end{equation}
Then (P2$^{\prime}$) and \eqref{EA-3} imply that
$\frac{\varepsilon}{\norm{y}} \norm{x}\le \zeta_{0}$ or that
$\norm{x} \le \frac{\zeta_{0}}{\varepsilon} \norm{y}$.
Therefore (P2$^{\prime}$) is equivalent to normality of the cone $P$.
\end{proof}

It is also the case that (P1)--(P2) are weaker than
\emph{uniform strong positivity} property which is defined as
\begin{itemize}
\item[(H1)]
There exists $\tau>0$ and $\xi>0$
such that $S_{\tau}x \succeq \xi\norm{x}\,\Hat{x}$ for all $x\in P$,
\end{itemize}
or in a seemingly weaker form as
\begin{itemize}
\item[(H1$^\prime$)]
For any compact subset $\cK\subset P$ there exists $\tau=\tau(\cK)>0$
and $\xi=\xi(\cK)>0$
such that $S_{\tau}x \succeq \xi\norm{x}\,\Hat{x}$ for all $x\in \cK$.
\end{itemize}

We first show that (H1) and (H1$^\prime$) are equivalent.

\begin{lemma}\label{LA-0}
\textup{(H1)} $\Longleftrightarrow$ \textup{(H1$^\prime$)}.
\end{lemma}

\begin{proof}
Obviously (H1) $\Longrightarrow$ (H1$^\prime$).

To prove the converse suppose (H1) does not hold.
Then there exists a sequence $\{x_{n}\}\subset P$ 
with $\norm{x_{n}}=1$ and a sequence $\tau_{n}\nearrow\infty$
such that $\ua(S_{\tau_{n}}x_{n})\searrow0$.
Hence $\ua(S_{\tau_{n}}x_{n})x_{n}\searrow0$,
so that the set $\{\ua(S_{\tau_{n}}x_{n})x_{n}\}$ is precompact.
Therefore by (H1$^\prime$) there exists $\tau>0$ and $\xi>0$ such that
$S_{\tau}(\ua(S_{\tau_{n}}x_{n})x_{n})\succeq \xi \ua(S_{\tau_{n}}x_{n})\Hat{x}$
which is equivalent (by 1-homogeneity) to
$S_{\tau}x_{n}\succeq \xi\Hat{x}$.
But $S_{\tau}x_{n}\succeq \xi\Hat{x}$ implies
that $\ua(S_{\tau}x_{n}) \ge \xi$.
Since $\ua(S_{\tau}x_{n})\le \ua(S_{\tau_{n}}x_{n})$ whenever $\tau_{n}\ge\tau$,
we obtain a contradiction with the property $\ua(S_{\tau_{n}}x_{n})\searrow0$.
Therefore (H1$^\prime$) cannot hold and the proof is complete.
\end{proof}

We need the following lemma.

\begin{lemma}\label{LA-1a}
Provided $\interior(P)\ne\varnothing$ then for every $x\in\Dot{P}$ there
exists $C_{0}=C_{0}(x)>0$ such that $y\succeq x$ implies $\norm{y}\ge C_{0}$.
\end{lemma}

\begin{proof}
Fix any $x_{0}\in\interior(P)$.
If the assertion in the lemma is not true there exists $\{y_{n}\}\subset P$
with $\norm{y_{n}}\searrow 0$ such that $y_{n}\succeq x$.
Then since $x_{0}\in\interior(P)$ there exists a sequence
$\varepsilon_{n}\searrow0$, such that $\varepsilon_{n}x_{0}\succeq y_{n}$.
But this implies $\varepsilon_{n} x_{0}\succeq x$ and taking limits as
$n\to\infty$ we have $0\succeq x$ which contradicts $x\in\Dot{P}$.
\end{proof}

We next show that uniform strong positivity implies (P1)--(P2).

\begin{lemma}\label{LA-2}
\textup{(H1)} $\Longrightarrow$  \textup{(P1)--(P2)}.
\end{lemma}

\begin{proof}
By (H1) we have
\begin{align}\label{EA-1}
S_{\tau}(\Hat{x} - z)+ S_{\tau}z &\;\succeq\; \xi\norm{\Hat{x} - z}\,\Hat{x}
+\xi\norm{z}\,\Hat{x}\\
&\;\succeq\; \xi\norm{\Hat{x}}\,\Hat{x}\,.\nonumber
\end{align}
By \eqref{EA-1} and Lemma~\ref{LA-1a} we have
\begin{align}\label{EA-2}
\norm{S_{\tau}(\Hat{x} - z)}+ \norm{S_{\tau}z}&\;\ge\;
\norm{S_{\tau}(\Hat{x} - z)+ S_{\tau}z}\\
&\;\ge\;
C_{0}\, \xi\norm{\Hat{x}}\,.\nonumber
\end{align}
It is clear that \eqref{EA-2} is stronger than (P1), since it
holds for any $z\preceq \Hat{x}$.

Next we show that (H1) $\Longrightarrow$ (P2).
By Lemma~\ref{LA-0} it is enough to show that
(H1$^\prime$) $\Longrightarrow$ (P2).
By the increasing property
$x\preceq \alpha \Hat{x}$ implies
$S_{\tau}x\preceq \alpha \Hat{x}$, which combined with (H1$^\prime$)
implies that $\xi\norm{x}\Hat{x}\preceq \alpha \Hat{x}$,
which in turn implies
$\norm{x}\le \xi^{-1}\alpha$.
\end{proof}

\subsection{The positive eigenpair of the Nisio semigroup}\label{S4.2}

We now return to the Nisio semigroup in \eqref{semi}.

\begin{lemma}\label{L-4.1}
There exists a unique pair
$(\rho,\varphi) \in \mathbb{R}\times C^{2}_{\gamma,+}(\Bar{Q})$
satisfying $\norm{\varphi}_{0;\Bar{Q}}=1$
such that
\begin{equation*}
S_{t}\varphi\;=\;e^{\rho t}\varphi\,, \quad t \ge 0\,.
\end{equation*}
The pair $(\rho, \varphi)$ is a solution to the p.d.e.
\begin{equation}\label{HJB}
\rho\,\varphi(x) \;=\; \cG \varphi(x) \;=\;
\inf_{v\in\cV}\;\bigl(\LA_{v}\varphi(x) + r(x,v)\varphi(x)\bigr)
\quad\text{in~} Q\,,
\quad \langle\nabla\varphi,\gamma\rangle = 0 \quad \text{on~} \partial{Q}\,,
\end{equation}
where \eqref{HJB} specifies $\rho$ uniquely in $\mathbb{R}$ and
$\varphi$, with $\norm{\varphi}_{0;\Bar{Q}}=1$, uniquely in
$C^{2}_{\gamma,+}(\Bar{Q})$.
\end{lemma}

\begin{proof}
It is clear that $S_{t}$ is superadditive.
If $f\in C^{2}_{\gamma,+}(\Bar{Q})$ then \eqref{stochasticrep} implies that
the solution $\psi$ of \eqref{pde} is non-negative.
Moreover by the strong maximum principle \cite[Theorem~3, p.~38]{Friedman}
and the Hopf boundary lemma \cite[Theorem~14, p.~49]{Friedman} it follows
that $\psi(t,\cdot\,)>0$ for all $t>0$.
Hence the strong positivity hypothesis in Corollary~\ref{C-4.2} is satisfied.
Since also the compactness hypothesis holds by Lemma~\ref{L-3.1},
the first statement follows by Corollary~\ref{C-4.2}.
That \eqref{HJB} holds follows from (7) of
Theorem~\ref{T-3.1} (see also \cite[pp.~73--75]{Bork}).
Uniqueness follows from the following argument.
Suppose $\Hat\rho\in\mathbb{R}$ and $\Hat\varphi\in C^{2}_{\gamma,+}(\Bar{Q})$
solve
\begin{equation*}
\Hat\rho\, \Hat\varphi(x) \,=\,
\inf_{v\in\cV}\;\bigl(\LA_{v}\Hat\varphi(x) + r(x,v)\Hat\varphi(x)\bigr)\,.
\end{equation*}
Then by direct substitution we have
\begin{align*}
\frac{\partial}{\partial{t}} \bigl(e^{\Hat\rho t}\, \Hat\varphi(x)\bigr)
&\;=\;\Hat\rho\,e^{\Hat\rho t}\, \Hat\varphi(x)\\[5pt]
&\,=\,\inf_{v\in\cV}\;
\bigl[\LA_{v}\bigl(e^{\Hat\rho t}\Hat\varphi(x)\bigr)
+ r(x,v)\bigl(e^{\Hat\rho t}\Hat\varphi(x)\bigr)\bigr]\,.
\end{align*}
Therefore,
$S_{t}\Hat\varphi = e^{\Hat\rho t}\, \Hat\varphi$, and by the uniqueness
assertion in Corollary~\ref{C-4.2} we have
$\Hat\rho=\rho$ and $\Hat\varphi = C\varphi$
for some positive constant $C$.
\end{proof}

\begin{remark}\label{R4.2}
Consider the operator
$R_{t} : C^{2+ \delta}_{\gamma}(\Bar{Q}) \to C^{2 + \delta}_{\gamma}(\Bar{Q})$
defined by $R_{t} f = - S_{t} (-f)$.
Then by same arguments as in the proof of Lemma~\ref{L-4.1} using
Corollary~\ref{C-4.2}, there exists a unique $\beta \in \mathbb{R}$ and
$\psi >0$ in  $C^{2 + \delta}_{\gamma}(\Bar{Q})$ such that
\begin{equation*}
R_{t} \psi \;= \; e^{\beta t} \psi\,.
\end{equation*}
Hence the pair $(e^{\beta t}, -\psi)$ is an eigenvalue-function pair of $S_{t}$.
Now the same arguments as in the proof of Lemma~\ref{L-4.1}
lead to the conclusion that $(\beta, \psi)$ is the unique positive solution pair of
\begin{equation*}
\beta\,\psi(x) \;=\;
\sup_{v\in\cV}\;\bigl(\LA_{v}\psi(x) + r(x,v)\psi(x)\bigr)
\quad\text{in~} Q\,,
\quad \langle\nabla\psi,\gamma\rangle = 0 \quad \text{on~} \partial{Q}\,,
\end{equation*}
Hence $(\beta, -\psi)$ is the unique solution pair of
\eqref{HJB} satisfying $-\psi < 0$.
Moreover it is easy to see that $\rho \le \beta$ and that $\beta$ is the
principal eigenvalue of both operators $R_{t}$, $S_{t}$.
This leads to the conclusion that the risk-sensitive control problem where
the controller tries to maximize the risk-sensitive cost (\ref{risk})
leads to the value $\beta$ which is the principal eigenvalue.
\end{remark}

\smallskip
\begin{remark}\label{Rem-4.1} The p.d.e.\ in \eqref{HJB} is the
Hamilton-Jacobi-Bellman equation
for the risk-sensitive control problem \cite{Bisw}.
\end{remark}

\smallskip
\begin{lemma}\label{L-4.2}
Let $\cM(\Bar{Q})$ denote the space of finite Borel measures on $\Bar{Q}$.
Then
\begin{equation*}
\bigl(C^{2}_{\gamma}(\Bar{Q})\bigr)^{*}_{+}\;=\;\cM(\Bar{Q})\,.
\end{equation*}
\end{lemma}

\begin{proof}
Let $\Lambda\in\bigl(C^{2}_{\gamma}(\Bar{Q})\bigr)^{*}_{+}$.
Then for $f\in C^{2}_{\gamma}(\Bar{Q})$ by positivity of $\Lambda$
we have
\begin{align*}
\bigl| \Lambda(f) \bigr| &\;=\;  \bigl| \Lambda(f + \|f\|_{0;\Bar{Q}}
\cdot {\bm 1} ) -\Lambda(\|f\|_{0;\Bar{Q}} \cdot {\bm 1} ) \bigr|\\[3pt]
&\;\le\;  \max \bigl\{ \Lambda(f + \|f\|_{0;\Bar{Q}} \cdot {\bm 1} ),
\Lambda(\|f\|_{0;\Bar{Q}} \cdot {\bm 1} )\bigr\}\\[3pt]
&\;\le\;  \Lambda(2 \|f\|_{0;\Bar{Q}} \cdot {\bm 1} )\\[3pt]
&\;=\;  2 \|f \|_{0;\Bar{Q}} \Lambda({\bm 1}) \,.
\end{align*}
It follows that $\Lambda$ is a bounded linear functional on the linear subspace
$C^{2}_{\gamma}(\Bar{Q})$ of $C(\Bar{Q})$.
By the Hahn-Banach theorem $\Lambda$ can be extended to some
$\psi\in \bigl(C(\Bar{Q})\bigr)^{*}$. Clearly $\psi$ is a positive linear functional.
By the Riesz representation theorem there exists $\mu\in\cM(\Bar{Q})$
such that $\psi(f) = \int_{\Bar{Q}} f\, d \mu$ for all
$f\in C(\Bar{Q})$.
Therefore $\Lambda(f) =  \int_{\Bar{Q}} f\,d \mu$ for all
$f\in C^{2}_{\gamma}(\Bar{Q})$.
This shows that $\bigl(C^{2}_{\gamma}(\Bar{Q})\bigr)^{*}_{+}\subset\cM(\Bar{Q})\,.$
It is clear that
$\cM(\Bar{Q})\subset\bigl(C^{2}_{\gamma}(\Bar{Q})\bigr)^{*}_{+}$,
so equality follows.
\end{proof}

\begin{lemma}\label{L-4.3}
Let $\delta\in(0,\beta_{0})$.
Then for any $f \in C^{2+\delta}_{\gamma,+}(\Bar{Q})$ we have
\begin{equation*}
\limsup_{t \downarrow 0}\;
\inf_{\substack{\mu\in{\mathcal M}(\Bar{Q})\\[1pt]\int\!f\,d\mu = 1}}\;
\int_{\Bar{Q}} \frac{S_{t} f (x) -f (x)}{t}\,\mu(dx)
\;=\;\inf_{\substack{\mu\in{\mathcal M}(\Bar{Q})\\[1pt]\int\!f\,d\mu = 1}}\;
 \int_{\Bar{Q}} \cG f (x)\,\mu(dx)
\end{equation*}
and
\begin{equation*}
\liminf_{t \downarrow 0}\;
\sup_{\substack{\mu\in{\mathcal M}(\Bar{Q})\\[1pt]\int\!f\,d\mu = 1}}\;
\int_{\Bar{Q}} \frac{S_{t} f (x) -f (x)}{t}\,\mu(dx) \;=\;
\sup_{\substack{\mu\in{\mathcal M}(\Bar{Q})\\[1pt]\int\!f\,d\mu = 1}}\;
\int_{\Bar{Q}} \cG f (x)\,\mu(dx)\,.
\end{equation*}
\end{lemma}

\smallskip
\begin{proof} Note that
\begin{equation*}
\lim_{t \downarrow 0}\;\frac{S_{t} f(x) - f(x)}{t} \;=\;
\cG f(x)\,, \quad x \in \Bar{Q}\,.
\end{equation*}
Hence using the dominated convergence theorem\footnote{
Note that
\begin{align*}
\biggl|\frac{S_{t} f(x) - f(x)}{t} \biggr| &\;\le\;
\inf_{v(\cdot)} \frac{1}{t} E_{x} \biggl[\int^t_{0} e^{\int^s_{0} r(X_z, v_z) dz}
\abs{\LA_{v_{s}} f (X_{s}) + r(X_{s}, v_{s}) f(X_{s})}\, ds \biggr]\\
&\;\le\; K e^{r_{\rm{max}}}\,,\quad 0 \le t \le 1\,,
\end{align*}
for some constant $K>0$.}, we obtain, for all $\mu \in {\mathcal M}(\Bar{Q})$ satisfying
$\int f d \mu =1$,
\begin{equation*}
\lim_{t \downarrow 0}\;\int_{\Bar{Q}} \frac{S_{t} f(x) - f(x)}{t} \mu(dx)\;=\;
\int_{\Bar{Q}} \cG f(x)\,\mu(dx)\,.
\end{equation*}
Therefore
\begin{align*}
\limsup_{t \downarrow 0}\;
\inf_{\substack{\tilde{\mu}\in{\mathcal M}(\Bar{Q})\\[1pt]
\int\!f\,d\tilde{\mu} = 1}}\;
\int_{\Bar{Q}} \frac{S_{t} f(x) - f(x)}{t} \tilde{\mu}(dx) &\;\le\;
\lim_{t \downarrow 0}\;\int_{\Bar{Q}} \frac{S_{t} f(x) - f(x)}{t}\,\mu(dx)\\
&\;=\;\int_{\Bar{Q}} \cG f(x)\,\mu(dx)
\end{align*}
for all $\mu \in {\mathcal M}(\Bar{Q})$ satisfying $\int f d \mu =1$.
Hence
\begin{equation}\label{appendixeq1}
\limsup_{t \downarrow 0}\;
\inf_{\substack{\mu\in{\mathcal M}(\Bar{Q})\\[1pt]\int\!f\,d\mu = 1}}\;
\int_{\Bar{Q}} \frac{S_{t} f(x) - f(x)}{t} \mu(dx)  \;\le \;
\inf_{\substack{\mu\in{\mathcal M}(\Bar{Q})\\[1pt]\int\!f\,d\mu = 1}}\;
\int_{\Bar{Q}}\cG f(x)\,\mu(dx)\,.
\end{equation}
Since for each $t > 0$ the map
$\mu \mapsto \int_{\Bar{Q}} \frac{S_{t} f(x) - f(x)}{t} \mu(dx)$
from ${\mathcal M}(\Bar{Q}) \to \mathbb{R}$ is continuous,
there exists a $\mu_{t} \in {\mathcal M}(\Bar{Q})$ satisfying
$\int f d\mu_{t} =1$ such that
\begin{equation*}
\inf_{\substack{\mu\in{\mathcal M}(\Bar{Q})\\[1pt]\int\!f\,d\mu = 1}}\;
\int_{\Bar{Q}} \frac{S_{t} f(x) - f(x)}{t}\,\mu(dx) \;=\;
\int_{\Bar{Q}} \frac{S_{t} f(x) - f(x)}{t}\,\mu_{t}(dx) \,.
\end{equation*}
Clearly $\{\mu_{t}\}$ is tight.
Let  $\Hat{\mu}$ be a limit point of $\mu_{t}$ as $t \to 0$.
Suppose $\mu_{t_n} \to \Hat{\mu}$ in ${\mathcal M}(\Bar{Q})$ as $t_n \downarrow 0$.
Then
$\int f d \Hat{\mu} =1$.
Note that for $f \in C^{2+\delta}_{\gamma,+}(\Bar{Q})$,
\begin{equation}\label{appendixeq2}
\frac{S_{t} f(x) - f(x)}{t} \;=\;\frac{1}{t} \int^t_{0}
\partial_{s} u^f(s, x)\,ds\,,
\end{equation}
with $u^f(t,\cdot\,)\df S_{t}f(\cdot)$.
By the H\"older continuity of $\partial_{s} u^f$
on $[0,\,1] \times \Bar{Q}$, there exists $k_{1}>0$ such that
\begin{equation}\label{appendixeq3}
|\partial_{s} u^f (s, x) - \partial_{s} u^f(s, y)| \;<\; k_{1} |x-y|^{\delta}
\qquad\forall x, y \in \Bar{Q}\,, \; s \in [0,\,1]\,.
\end{equation}
Therefore by \eqref{appendixeq2} and \eqref{appendixeq3}
$x \mapsto \frac{S_{t} f(x) - f(x)}{t}$ is H\"older equicontinuous over
$t \in (0,\,1]$, and the convergence
\begin{equation*}
\lim_{t \downarrow 0}\;\frac{S_{t} f(x) - f(x)}{t}  \;=\;{\mathcal G}f(x)
\end{equation*}
is uniform in $\Bar{Q}$.
Hence from
\begin{multline*}
\int_{\Bar{Q}} \frac{S_{t_n} f(x) - f(x)}{t_n}\,\mu_{t_n}(dx) \;=\;
\int_{\Bar{Q}} \left(\frac{S_{t_n} f(x) - f(x)}{t_n} - {\mathcal G} f (x)\right)
\,\mu_{t_n}(dx)\\[3pt]
 + \int_{\Bar{Q}} {\mathcal G}f(x)\,\mu_{t_n}(dx)
\end{multline*}
it follows that
\begin{align*}
\lim_{n \to \infty}\;
\int_{\Bar{Q}}\left(\frac{S_{t_n} f(x) - f(x)}{t_n} \right)\,\mu_{t_n}(dx)
&\;=\;\int_{\Bar{Q}} \cG f (x)\,\Hat{\mu}(dx) \\[5pt]
&\;\ge\; \inf_{\substack{\mu\in{\mathcal M}(\Bar{Q})\\[1pt]\int\!f\,d\mu = 1}}\;
\int_{\Bar{Q}} \cG f(x)\,\mu(dx)\,.
\end{align*}
Hence
\begin{equation}\label{appendixeq6}
\limsup_{t \downarrow 0}
\int_{\Bar{Q}}  \left(\frac{S_{t} f(x) - f(x)}{t}\right)\,\mu_{t}(dx)
\;\ge\;  \inf_{\substack{\mu\in{\mathcal M}(\Bar{Q})\\[1pt]\int\!f\,d\mu = 1}}\;
\int_{\Bar{Q}} \cG f(x)\, \mu(dx)\,.
\end{equation}
From \eqref{appendixeq1} and \eqref{appendixeq6}, the result follows.
The proof of the second limit follows by a symmetric argument.
\end{proof}

We next prove the main result.

\begin{proof}[Proof of Theorem~\ref{T-4.4}]
Let $\delta\in(0, \beta_{0})$.
Since $\rho\, \varphi=\cG\varphi$ by Lemma~\ref{L-4.1}, we obtain
\begin{align*}
\rho &\;=\;
\inf_{\substack{\mu\in{\mathcal M}(\Bar{Q})\\[1pt]\int\!\varphi\,d\mu = 1}}\;
\int\cG \varphi\,d\mu \\
&\;\le\; \sup_{f \in C^{2+\delta}_{\gamma,+}(\Bar{Q})}\;
\inf_{\substack{\mu\in{\mathcal M}(\Bar{Q})\\[1pt]\int\!f\,d\mu = 1}}\;
\int\cG f\,d\mu\,.
\end{align*}
To show the reverse inequality we use
Theorem~\ref{T-4.2} and Lemma~\ref{L-4.2}.
We have
\begin{equation*}
e^{\rho t} \;=\;  \sup_{g \in C^{2+\delta}_{\gamma,+}(\Bar{Q})}\;
\inf_{\substack{\mu\in{\mathcal M}(\Bar{Q})\\[1pt]\int\!g\,d\mu = 1}}\;
\int S_{t}g\,d\mu\,.
\end{equation*}
Therefore, using Lemma~\ref{L-4.3} we obtain
\begin{align*}
\rho &\;=\;
\lim_{t\downarrow 0}\;
\sup_{g \in C^{2+\delta}_{\gamma,+}(\Bar{Q})}\;
\inf_{\substack{\mu\in{\mathcal M}(\Bar{Q})\\[1pt]\int\!g\,d\mu = 1}}\;
\int \frac{S_{t}g - g}{t}\,d\mu \\
&\;\ge\; \limsup_{t \downarrow 0}\;
\inf_{\substack{\mu\in{\mathcal M}(\Bar{Q})\\[1pt]\int\!f\,d\mu = 1}}\;
\int \frac{S_{t}f - f}{t}\;d\mu\\
&\;=\;  \inf_{\substack{\mu\in{\mathcal M}(\Bar{Q})\\[1pt]\int\!f\,d\mu = 1}}\;
\int\cG f\,d\mu\,
\end{align*}
for all $f \in C^{2+\delta}_{\gamma,+}(\Bar{Q})$.
Therefore,
\begin{equation*}
\rho \;\ge \; \sup_{f \in C^{2+\delta}_{\gamma,+}(\Bar{Q})}\;
\inf_{\substack{\mu\in{\mathcal M}(\Bar{Q})\\[1pt]\int\!f\,d\mu = 1}}\;
\int\cG f\, d\mu\,.
\end{equation*}
Using a symmetric argument to establish the first equality in
\eqref{ET-4.4c} below,
we obtain
\begin{align}\label{ET-4.4c}
\rho &\;=\; \inf_{f \in C^{2+\delta}_{\gamma,+}(\Bar{Q})}\;
\sup_{\substack{\mu\in{\mathcal M}(\Bar{Q})\\[1pt]\int\!f\,d\mu = 1}}\;
\int\cG f\,d\mu\\
&\;=\; \sup_{f \in C^{2+\delta}_{\gamma,+}(\Bar{Q})}\;
\inf_{\substack{\mu\in{\mathcal M}(\Bar{Q})\\[1pt]\int\!f\,d\mu = 1}}\;
\int\cG f\,d\mu\nonumber
\end{align}
for all $\delta\in(0,\beta_{0})$.
Note that the outer `$\inf$' and `$\sup$' in \eqref{ET-4.4c} are realized
at the function $\varphi$ in Lemma~\ref{L-4.1}.
Therefore, since $\varphi>0$, equation \eqref{ET-4.4c} remains valid
if we restrict the outer `$\inf$' and `$\sup$' on $f>0$.
Hence using the probability measure $d\nu=f\,d\mu$ we can write
\eqref{ET-4.4c} as
\begin{align*}
\rho &\;=\; \inf_{f \in C^{2+\delta}_{\gamma,+}(\Bar{Q}),\,f>0}\;
\sup_{\nu \in \PA(\Bar{Q})}\;\int\frac{\cG f}{f}\,d\nu\\
&\;=\; \sup_{f \in C^{2+\delta}_{\gamma,+}(\Bar{Q}),\,f>0}\;
\inf_{\nu \in \PA(\Bar{Q})}\;\int\frac{\cG f}{f}\,d\nu\,.
\end{align*}
Therefore
\begin{equation}\label{ET-4.4d}
\inf_{f \in C^{2}_{\gamma,+}(\Bar{Q}),\,f>0}\;
\sup_{\nu \in \PA(\Bar{Q})}\;\int\frac{\cG f}{f}\,d\nu
\;\le\;\rho \;\le\;\sup_{f \in C^{2}_{\gamma,+}(\Bar{Q}),\,f>0}\;
\inf_{\nu \in \PA(\Bar{Q})}\;\int\frac{\cG f}{f}\,d\nu\,.
\end{equation}
Suppose that the  inequality on the r.h.s.\ of \eqref{ET-4.4d} is strict.
Then for some $\Hat{f} \in C^{2}_{\gamma,+}(\Bar{Q})$ we have
\begin{equation*}
\inf_{\nu \in \PA(\Bar{Q})}\;\int\frac{\cG \Hat{f}}{\Hat{f}}\,d\nu \;>\; \rho\,.
\end{equation*}
Since $\cG:C^{2}_{\gamma,+}(\Bar{Q})\to C^{0}(\Bar{Q})$
is continuous and since $C^{2+\delta}_{\gamma,+}(\Bar{Q})$ is dense in
$C^{2}_{\gamma,+}(\Bar{Q})$
in the $\norm{\,\cdot\,}_{2;\Bar{Q}}$ norm,
there exists $g\in C^{2+\delta}_{\gamma,+}(\Bar{Q})$, $g>0$, such that
$\min_{\Bar{Q}}\;\frac{\cG g}{g}>\rho$.
However this contradicts Theorem~\ref{T-4.2} which means that the first equality
in \eqref{ET-4.4b} must hold.
The proof of the second equality in \eqref{ET-4.4b} is similar.
The last assertion of the theorem follows via the change of measure
$f\,d\mu= d\nu$.
\end{proof}

\begin{remark}
As pointed out in the proof of
Theorem~\ref{T-4.4} the outer `$\inf$', resp.\ `$\sup$' in \eqref{ET-4.4b}
and \eqref{ET-4.4c} are in fact `$\min$', `$\max$' attained by $\varphi$.
\end{remark}

Concerning the stability of the semigroup we have the following lemma.

\begin{lemma}\label{L-4.4}
There exist $M>0$ and $\theta>0$ such that
for any $f\in C^{2}_{\gamma,+}(\Bar{Q})$ we have
\begin{equation*}
\bnorm{e^{-\rho t} S_{t} f -\alpha^{*}(f)\varphi}_{0;\Bar{Q}}
\;\le\; M e^{-\theta t}\norm{f}_{0;\Bar{Q}}\qquad \forall t\ge1\,,
\end{equation*}
for some $\alpha^{*}(f)\in\RR_{+}$.
\end{lemma}

\begin{proof}
Without loss of generality we assume $\varrho=0$.
We first verify that property (P1) of Theorem~\ref{T-general} holds.
Let $\tau=\nicefrac{1}{2}$.
We claim that there exists a constant $c_{0} > 0$ such that
\begin{equation}\label{E-transition}
\bigl(E^{v}_{x}[f(X_{\tau})]\bigr)^{2} \;\le\;c_{0}\, E^{v'}_{x}[f(X_{\tau})]
\qquad
 \forall  f \in C(\Bar{Q})\,,\; 0\le f\le\varphi\,,
\end{equation}
and for all Markov controls $v$, $v'$ and $x \in \Bar{Q}$.
The proof of (\ref{E-transition}) is as follows.
To distinguish between processes, let $Y$, $Z$ denote the
processes corresponding to the controls $v$, $v'$ respectively.
Then using Girsanov's theorem,
it follows that
if we define
\begin{equation*}
F(\tau) \;\df\;
\int_{0}^{\tau} \sigma^{-1}(Y_{t})
[b(Y_{t}, v_{t}) - b(Y_{t}, v'_{t})] dW_{t} - \frac{1}{2} \int_{0}^{\tau}
\norm{\sigma^{-1}(Y_{t})[b(Y_{t}, v_{t}) - b(Y_{t}, v'_{t})]}^{2}dt\,,
\end{equation*}
then
\begin{align*}
E_{x} [f(Y_{\tau})] &\;=\;
E_{x}\bigl[e^{F(\tau)}\, f(Z_{\tau}) \bigr]\\[3pt]
&\;\le\;
\bigl(E_{x}\bigl[f^{2}(Z_{\tau})\bigr]\bigr)^{\nicefrac{1}{2}}
\bigl(E_{x}\bigl[e^{2F(\tau)}
\bigr]\bigr)^{\nicefrac{1}{2}}\\[3pt]
&\;\le\; \bigl(E_{x}\bigl[f^{2}(Z_{\tau})\bigr]\bigr)^{\nicefrac{1}{2}}
\bigl(E_{x}\bigl[e^{\int_{0}^{\tau}
\norm{\sigma^{-1}(Y_{t})[b(Y_{t}, v_{t}) - b(Y_{t}, v'_{t})]}^{2}dt}
\bigr)^{\nicefrac{1}{2}}\\[3pt]
&\;\le\; c_{1} \bigl(E_{x}\bigl[f^{2}(Z_{\tau})\bigr]\bigr)^{\nicefrac{1}{2}}
\\[3pt]
&\;\le\;
c_{1} \norm{\varphi}_{0;Q}^{\nicefrac{1}{2}}
\bigl(E_{x}\bigl[f(Z_{\tau})\bigr]\bigr)^{\nicefrac{1}{2}}
\end{align*}
where $c_{1}> 0$ is a constant which only depends on the bounds of
$\sigma^{-1}$ and $b$.
This proves \eqref{E-transition}.
For $f \in C(\Bar{Q})$ satisfying $0 \le f \le \varphi$ and for
any fixed $v$ we have
\begin{align}\label{E-inequality2}
S_{\tau}(\varphi- f) (x) &\;\ge\; e^{r_{\rm{min}}} E^{v_{1}}_{x}
\bigl[\varphi(X_{\tau})- f(X_{\tau})\bigr]\\[5pt]
&\;\ge\; e^{r_{\rm min}} c_{0}^{-1}
\bigl(E^{v}_{x} \bigl[\varphi(X_{\tau})- f(X_{\tau})\bigr]\bigr)^{2}\nonumber
\end{align}
and
\begin{align}\label{E-inequality3}
S_{\tau}(f)(x) &\;\ge\; e^{r_{\rm{min}}}E^{v_{2}}_{x}\bigl[f(X_{\tau})\bigr]\\[5pt]
&\;\ge\; e^{r_{\rm{min}}} c_{0}^{-1}
\bigl(E^{v}_{x}\bigl[f(X_{\tau})\bigr]\bigr)^{2}\,,\nonumber
\end{align}
where $v_{1}$, $v_{2}$ are  the corresponding minimizers.
Note that\footnote{The first part of the inequality below follows from the fact that
$(a-x)^{2} + x^{2}$, $0 \le x \le a$
attains it minimum at $x = \frac{a}{2}$}
\begin{equation}\label{E-inequality4}
\bigl(E^{v}_{x}\bigl[\varphi(X_{\tau})- f(X_{\tau})\bigr]\bigr)^{2}
+ \bigl(E^{v}_{x}\bigl[f(X_{\tau})\bigr]\bigr)^{2}
\;\ge\; \frac{1}{2}\bigl(E^{v}_{x}\bigl[\varphi(X_{\tau})\bigr]\bigr)^{2}
\;\ge\; \frac{1}{2} \bigl(\min \varphi\big)^{2}\,.
\end{equation}
Adding \eqref{E-inequality2} and \eqref{E-inequality3} and
using  \eqref{E-inequality4}, it follows that
\begin{equation*}
\norm{S_{\tau}(\varphi-f)}+\norm{S_{\tau} f}
\;\ge\; \frac{e^{r_{\rm{min}}}}{2c_{0}}\bigl(\min \varphi\big)^{2}\,,
\end{equation*}
which establishes property (P1).
On the other hands, property (P2) of Theorem~\ref{T-general}
is trivially satisfied  under the $\norm{\,\cdot\,}_{0;\Bar{Q}}$ norm.
Hence the result follows by Theorem~\ref{T-general}~(iii).
\end{proof}

\subsection{The Donsker--Varadhan functional}

Let $U =\{u\}$, i.e., a singleton, and $v(\cdot) \equiv v \df \delta_{u}$,
thus reducing the problem to an uncontrolled one.
Thus $\cG = \LA_{v} + r(x,v)$ is a linear operator.
By \cite[Lemma~2, pp.~781--782]{DoVa}, the first equality in \eqref{ET-4.4b}
equals the Donsker--Varadhan functional
\begin{equation*}
\sup_{\nu \in \PA(\Bar{Q})}\left(\int_{\Bar{Q}} r(x,v)\,\nu(dx)
- I(\nu)\right)\,,
\end{equation*}
where
\begin{equation*}
I(\nu) \;\df\; - \inf_{f \in C^{2}_{\gamma,+}(\Bar{Q}),\, f > 0}\;
\int\frac{\LA_{v} f}{f}\,d\nu\,.
\end{equation*}
More generally if $r(x,v)$ does not depend on $v$,
say $r(x,v)=r(x)$ and $\mathcal{A}$ is defined by
\begin{equation*}
\mathcal{A} f(x)\;\df\;\frac{1}{2}\trace\left(a(x)\nabla^{2}f(x)\right)
+ \min_{v\in\cV}\; \bigl[\langle b(x,v), \nabla f(x)\rangle\bigr]\,,
\end{equation*}
then
\begin{align*}
\rho &\;=\;
\sup_{\nu \in \PA(\Bar{Q})}\left(\int_{\Bar{Q}} r(x)\,\nu(dx)
- I(\nu)\right)\,,\\[5pt]
I(\nu) &\;=\; - \inf_{f \in C^{2}_{\gamma,+}(\Bar{Q}),\, f > 0}\;
\int\frac{\mathcal{A} f}{f}\,d\nu\,.
\end{align*}
This also takes the form
\begin{align*}
\rho &\;=\;
\sup_{x \in\Bar{Q}}\left(r(x)
- \Tilde{I}(x)\right)\,,\\[5pt]
\Tilde{I}(x) &\;\df\; - \inf_{f \in C^{2}_{\gamma,+}(\Bar{Q}),\, f > 0}\;
\frac{\mathcal{A} f(x)}{f(x)}\,.
\end{align*}
Our results thus provide a counterpart of the Donsker--Varadhan functional
for the nonlinear case arising from control.

It is also interesting to consider the substitution $f = e^{\psi}$. Then we obtain
\begin{align*}
\rho &\;=\; \inf_{\psi \in C^{2}_{\gamma}(\Bar{Q})}\;\sup_{\nu \in \PA(\Bar{Q})}\;
\int\inf_{v\in\cV}\;\sup_{w\in\mathbb{R}^{d}}\;
\Bigl(r(\,\cdot\,,v) - \frac{1}{2}\|w\|^{2} + \LA_{v}\psi +
\langle\nabla\psi, \sigma w\rangle\Bigr)\,d\nu\\
&\;=\; \sup_{\psi \in C^{2}_{\gamma}(\Bar{Q})}\; \inf_{\nu \in \PA(\Bar{Q})}\;
\int\inf_{v\in\cV}\;\sup_{w\in\mathbb{R}^{d}}\;
\Bigl(r(\,\cdot\,,v) - \frac{1}{2}\|w\|^{2} + \LA_{v}\psi +
\langle\nabla\psi, \sigma w\rangle\Bigr)\,d\nu\\
&\;=\; \inf_{\psi \in C^{2}_{\gamma}(\Bar{Q})}\;\sup_{\nu \in \PA(\Bar{Q})}\;
\int\sup_{v\in\cV}\;\inf_{w\in\mathbb{R}^{d}}\;
\Bigl(r(\,\cdot\,,v) - \frac{1}{2}\|w\|^{2} + \LA_{v}\psi +
\langle\nabla\psi, \sigma w\rangle\Bigr)\,d\nu\\
&\;=\; \sup_{\psi \in C^{2}_{\gamma}(\Bar{Q})}\; \inf_{\nu \in \PA(\Bar{Q})}\;
\int\sup_{v\in\cV}\;\inf_{w\in\mathbb{R}^{d}}\;
\Bigl(r(\,\cdot\,,v) - \frac{1}{2}\|w\|^{2} + \LA_{v}\psi +
\langle\nabla\psi, \sigma w\rangle\Bigr)\,d\nu\,,
\end{align*}
where the last two expressions follow from the standard Ky Fan min-max theorem
\cite{Fan}.
This is the standard logarithmic transformation to convert the
Hamilton-Jacobi-Bellman equation
for risk-sensitive control to the Hamilton-Jacobi-Isaacs equation for an
associated zero sum ergodic
stochastic differential game \cite{FlMc}, given by
\begin{equation}\label{HJI}
\inf_{v\in\cV}\;\sup_{w\in\mathbb{R}^{d}}\;
\Bigl(r(\,\cdot\,,v) - \frac{1}{2}\|w\|^{2} + \LA_{v}\psi +
\langle\nabla\psi, \sigma w\rangle\Bigr)  \;=\; \rho
\end{equation}
in $Q$, with $\langle\nabla\psi, \gamma \rangle = 0$ on $\partial{Q}$.
The expressions above bear the same relationship with \eqref{HJI} as what
Lemma~\ref{L-4.1} and Remark~\ref{Rem-4.1} spell out for \eqref{HJB}.

\section{Risk-sensitive control with periodic coefficients}
In this section we consider risk-sensitive control with periodic coefficients.
Consider a  controlled diffusion $X(\cdot)$ taking values in
$\mathbb{R}^d$ satisfying
\begin{equation}\label{sdeperiodic}
dX(t) \;=\; b(X(t), v(t))\,dt + \sigma(X(t))\,dW(t)
\end{equation}
for $t \ge 0$, with $X(0) = x$.

We assume that
\begin{enumerate}
\item
The functions $b(x,v)$,  $\sigma(x)$ and the running cost
$r(x,v)$ are periodic in $x_i$, $i =1,2,\dotsc,d$.
Without loss of generality we assume that the period equals $1$.
\item
$b : \mathbb{R}^d\times \cV \to \mathbb{R}^d$
is continuous and Lipschitz in its first argument
uniformly with respect to the second,
\item
$\sigma: \RR^{d} \to \mathbb{R}^{d\times d}$
is continuously differentiable, its derivatives are H\"older continuous
with exponent $\beta_{0}>0$, and is non-degenerate,
\item
$r : \RR^{d}\times\cV \to \mathbb{R}$ is continuous and Lipschitz in its
first argument uniformly with respect to the second.
We let
$r_{\rm{max}}\;\df\;\max_{(x,v)\in\Bar{Q}\times\cV}|r(x,v)|$.
\end{enumerate}
Admissible controls are defined as in (e).

We consider here as well the infinite horizon risk-sensitive problem which
aims to minimize the cost
in \eqref{risk} under the controlled process governed by \eqref{sdeperiodic}.
Recall the notation defined in Section~\ref{S2} and note that
$C^{0}(\RR^{d})$ is the space of all continuous and bounded
real-valued functions on $\RR^{d}$.
We define the semigroups of operators $\{S_{t}\,,\;t\ge0\}$
and $\{T^{u}_{t}\,,\;t\ge0\}$ acting
on $C^{0}(\RR^{d})$ as in \eqref{semi}--\eqref{E-Tu} relative to
the controlled process governed by \eqref{sdeperiodic}.
Also the operators $\LA_{v} : C^{2}(\RR^{d})\to C^{0}(\RR^{d})$ are as defined in
\eqref{E-Lu}.

Let $C_{p}(\mathbb{R}^d)$ denote the set of all $C^{0}(\mathbb{R}^d)$ functions
with period $1$ and in general if $\mathcal{X}$ is a subset of
$C^{0}(\mathbb{R}^d)$ we let
$\mathcal{X}_{p}(\mathbb{R}^d)\df  \mathcal{X}\cap C_{p}(\mathbb{R}^d)$.

We start with the following theorem which is analogous to Theorem~\ref{T-3.1}.
\begin{theorem}\label{T-5.1}
$\{S_{t}\,,\; t \ge 0\}$ acting on $C^{0}(\RR^{d})$
satisfies the following properties:
\begin{enumerate}
\item
Boundedness: $\|S_{t}f\|_{0;\RR^{d}} \le e^{r_{\rm{max}}t}\|f\|_{0;\RR^{d}}$.
Furthermore, $S_{t}\bm1 \ge  e^{r_{\rm{min}}t}\bm1$,
where $\bm1$ is the constant function $\equiv 1$.
\smallskip
\item
Semigroup property: $S_{0} = I$, $S_{t}\circ S_{s} = S_{t + s}$ for $s, t \ge 0$.
\smallskip
\item
Monotonicity: $f \ge$ (resp., $>$) $g \;\Longrightarrow\; S_{t}f \ge$
(resp., $>$) $S_{t}g$.
\smallskip
\item
Lipschitz property: $\|S_{t}f - S_{t} g\|_{0;\RR^{d}} \le
e^{r_{\rm{max}}t}\|f - g\|_{0;\RR^{d}}$.
\smallskip
\item
Strong continuity: $\|S_{t}f - S_{s}f\|_{0;\RR^{d}} \to 0$ as $t \to s$.
\smallskip
\item
Envelope property: $T^u_{t}f \ge S_{t}f$ for all $u \in U$ and
$S_{t}f \ge S_{t}'f$ for any
other $\{S_{t}'\}$ satisfying  this along with the foregoing properties.
\smallskip
\item
Generator: the infinitesimal generator of
$\{S_{t}\}$ is given by \eqref{E-gen}.
\smallskip
\item
For  $f \in  C_{p}(\mathbb{R}^d)$, $S_{t} f \in C_{p}(\mathbb{R}^d), t \ge 0$.
\end{enumerate}
\end{theorem}

\begin{proof}
Properties (1)--(4) and (6) follow by standard arguments
from \eqref{semi} and the bound on $r$.
That $S_{t}: C^{0}(\RR^{d})\to C^{0}(\RR^{d})$ is well known.
See Remark~\ref{R-5.1} below.
Property (8) follows from \eqref{semi} and the periodicity of the
data.
\end{proof}

\begin{theorem}\label{T-5.2}
For $f \in C^{2+\delta}_{p}(\mathbb{R}^d)$, $\delta\in(0,\beta_{0})$, the p.d.e.
\begin{equation}\label{pdeperiodic}
\frac{\partial}{\partial t}u(t,x) \;=\;
\inf_{v\in\cV}\;\bigl(\LA_{v} u(t,x) + r(x,v)u(t,x)\bigr)
\quad\text{in~} \RR_{+}\times\RR^{d}\,,
\end{equation}
with $u(0, x) = f(x)$ $\forall x \in \mathbb{R}^d$ has a unique solution in
$C^{1+\nicefrac{\delta}{2},2+\delta}_{p}\bigl([0,T] \times \mathbb{R}^d\bigr)$,
$T >0$.
The solution $\psi$ has the stochastic representation
\begin{equation}\label{E-str-per}
u(t,x) \;=\; \inf_{v(\cdot)}\;E_{x} \left[e^{\int^t_{0} r(X(s), v(s))\,ds}
 f(X(t))\right] \qquad \forall (t,x)\in[0,\infty)\times\RR^{d} \,.
\end{equation}
Moreover, for some $K_{T} > 0$ depending on $T$, $\delta$,
$\norm{f}_{2+\delta;\RR^{d}}$ and the bounds on the data, we have
\begin{equation*}
\norm{u}_{1+\nicefrac{\delta}{2},2+\delta;[0,T]\times B_{R}}\le K_{T}\,.
\end{equation*}
\end{theorem}

\begin{proof}
Without loss of generality we assume that $f$ is nonnegative.
Consider the p.d.e.
\begin{equation*}
\frac{\partial}{\partial t} u^R(t,x) \;=\; \inf_{v}\;\bigl(\LA_{v} u^R(t,x)
+ r(x, v) u^R(t,x)\bigr)
\quad\text{in~} \RR_{+}\times B_{R}\,,
\end{equation*}
with $u^R \,=\,0$ on $\RR_{+}\times\partial B_R$ and
with $u^R(0, x) =  f(x)g(R^{-1}x)$ for all $x \in B_R$,
where $g$ is a smooth non-negative, radially nondecreasing function which equals
$1$ on $\Bar{B}_{\frac{1}{2}}$ and $0$ on $B^c_{\frac{3}{4}}$.
From \cite[Theorem~6.1, pp. 452--453]{Lady},
the p.d.e.\ \eqref{pdeperiodic} has a unique solution $u^R$ in
$C^{1+\nicefrac{\delta}{2},2+\delta}\bigl([0,T] \times \Bar B_{R}\bigr)$, $T >0$.
This solution has the stochastic representation
\begin{equation*}
u^{R}(t,x) \;=\; \inf_{v(\cdot)}\;E_{x} \left[e^{\int^{t\wedge\tau_{R}}_{0}
r(X(s), v(s))\,ds}
f(X(t\wedge\tau_{R}))g(R^{-1}X(t\wedge\tau_{R}))\right]
\end{equation*}
for all $(t,x)\in[0,\infty)\times\RR^{d}$,
where $\tau_{R}$ denotes the first exit time from the ball $B_{R}$.
Clearly then $R\mapsto u^{R}$ is nondecreasing.
By \cite[Theorem~5.2, p.~320]{Lady} for each $T>0$
there exists a constant $K_{T}$ such that
\begin{equation*}
\norm{u^{R}}_{1+\nicefrac{\delta}{2},2+\delta;[0,T]\times B_{R}}\le K_{T}\,.
\end{equation*}
Therefore $u^{R}$ converges to a function
$u\in C^{1+\nicefrac{\delta}{2},2+\delta}\bigl([0,T] \times \Bar \RR^{d}\bigr)$,
as $R\to\infty$, which satisfies \eqref{pdeperiodic}--\eqref{E-str-per}.
The periodicity of $u(t,x)$ in $x$ follows by \eqref{E-str-per}
and the periodicity of the coefficients.
\end{proof}

\begin{remark}\label{R-5.1}
The regularity of the initial condition $f$ is only needed to obtain
continuous second derivatives at $t=0$.
It is well known that for each $f \in C^{0}(\mathbb{R}^d)$
\eqref{pdeperiodic} has a solution in
$C\bigl([0,T] \times \mathbb{R}^d\bigr)\cap
C_{\rm{loc}}^{1+\nicefrac{\delta}{2},2+\delta}\bigl((0,T)
\times \mathbb{R}^d\bigr)$, for $T>0$.
\end{remark}

\begin{theorem}\label{T-5.3}
There exists a unique $\rho \in \mathbb{R}$ and a $\varphi > 0$ in
$C^{2}_{p}(\mathbb{R}^d)$  unique up to a scalar multiple such that
\begin{equation*}
S_{t}\varphi\;=\;e^{\rho t}\varphi\,, \quad t > 0\,.
\end{equation*}
\end{theorem}

\begin{proof}
Using Theorem~\ref{T-5.2}, one can show as in the proof of
Lemma~\ref{L-3.1} that
$S_{t} : C^{2}_{p}(\mathbb{R}^d) \to C^{2}_{p}(\mathbb{R}^d)$ is compact for
each $t \ge 0$.
Now with  $\cX= C^{2}_{p}(\mathbb{R}^d)$ and
$P = \{f \in C^{2}_{p}(\mathbb{R}^d) : f \ge 0\}$
and $T = S_{t}$ for some $t \ge 0$, the conditions of Theorems~\ref{T-4.1} and
\ref{T-4.2} are easily verified using Theorem~\ref{T-5.1}.
Repeating the same argument as in the proof of Corollary~\ref{C-4.2},
completes the proof.
\end{proof}

\begin{lemma}
The pair $(\rho, \varphi)$ given in Theorem~\ref{T-5.3} is a solution to
the p.d.e.
\begin{equation}\label{HJBtorus}
\rho\,\varphi(x) \;=\;\inf_{v}\bigl(\LA_{v}\varphi(x) + r(x, v)\varphi(x)\bigr),
\end{equation}
where \eqref{HJBtorus} specifies $\rho$ uniquely in $\mathbb{R}$ and
$\varphi$ uniquely in $C^{2}_{p}(\mathbb{R}^d)$ up to a scalar multiple.
Moreover, $\inf_{\mathbb{R}^d} \varphi > 0$.
\end{lemma}

\begin{proof}
The proof is directly analogous to that of Lemma~\ref{L-4.1}.
\end{proof}

\begin{lemma}
$(C^{2}_{p}(\mathbb{R}^d))^* \simeq {\mathcal M}(Q)$, with $Q = [0, 1)^d$.
\end{lemma}

\begin{proof}
Let $\pi$ denote the projection of $\mathbb{R}^d$ to $[0, 1)^d$.
Set
\begin{equation*}
\mathcal{D}\;=\;\{f\circ\pi \in C(Q) : f \in C_{p} (\mathbb{R}^d) \}\,.
\end{equation*}
Then $\mathcal{D}$ is a linear subspace of $C^{0}(Q)$.

For $\Lambda \in (C_{p}(\mathbb{R}^d))^*$, define the linear map
$\tilde{\Lambda}: {\mathcal D}  \to \mathbb{R}$ by
\begin{equation*}
 \tilde{\Lambda}(f\circ\pi) \;=\;\Lambda (f) .
\end{equation*}
Then
\begin{equation*}
| \tilde{\Lambda}(f\circ\pi)| \;\le\;
\|\Lambda\| \|f\|_{0;\RR^{d}} \;\le\;
\|\Lambda\| \|f\circ\pi\|_{0;Q}\,.
\end{equation*}
i.e. $\tilde{\Lambda} \in {\mathcal D}^*$.
Using the Hahn-Banach theorem, there exists a continuous linear extension
$\Lambda' : C^{0}(Q) \to \mathbb{R}$
of $\tilde{\Lambda}$ such that $\|\Lambda'\| = \|\tilde{\Lambda}\|$.

Since $\bigl(C^{0}(Q)\bigr)^* = \mathcal{M}(Q),$ the set of all finite
signed Radon measures, we have   $(C_{p}(\mathbb{R}^d))^* \subseteq \mathcal{M}(Q)$.
The reverse inequality follows easily.
Hence $(C_{p}(\mathbb{R}^d))^* = \mathcal{M}(Q)$. Now the analogous argument
in Lemma~\ref{L-4.2}
can be used to complete the proof.
\end{proof}

Now by closely mimicking the proofs of Lemma~\ref{L-4.3} and Theorem~\ref{T-4.4},
we have

\begin{theorem}
$\rho$ satisfies
\begin{align*}
\rho &\;=\; \inf_{f \in C^{2}_{+}(Q)\cap\mathcal{D}}\;
\sup_{\mu\in\mathcal{M}(Q)\,:\,\int\! f\,d\mu=1}\;\int\cG f\,d\mu\\[3pt]
&\;=\; \sup_{f \in C^{2}_{+}(Q)\cap\mathcal{D}} \;
\inf_{\mu\in\mathcal{M}(Q)\,:\,\int\! f\,d\mu=1}\;\int\cG f\,d\mu,
\end{align*}
where ${\mathcal G}$ given in Theorem~\ref{T-5.1}.
\end{theorem}

The stability of the semigroup also follows as in Lemma~\ref{L-4.4}.
It is well known that  \eqref{sdeperiodic} has a transition probability
density $p(t,x,y)$ which is bounded away from zero, uniformly over
all Markov controls $v$, for $t=1$ and $x$, $y$ in a compact set.
It is straightforward to show that this implies property (P1).
Therefore exponential convergence follows by Theorem~\ref{T-general}~(iii).

\def\cprime{$'$}

\end{document}